\documentclass[pdflatex,sn-mathphys-num]{sn-jnl}

\usepackage[utf8]{inputenc}

\usepackage{graphicx}%
\usepackage{multirow}%
\usepackage{amsmath,amssymb,amsfonts}%
\usepackage{amsthm}%
\usepackage{mathrsfs}%
\usepackage[title]{appendix}%
\usepackage{xcolor}%
\usepackage{textcomp}%
\usepackage{manyfoot}%
\usepackage{booktabs}%
\usepackage{algorithm}%
\usepackage{algorithmicx}%
\usepackage{algpseudocode}%
\usepackage{listings}%
\usepackage{bbm,hyperref}
\usepackage{enumerate}
\usepackage{diagbox}
\usepackage{stmaryrd}

\newcommand{\po}{\left(}
\newcommand{\pf}{\right)}
\newcommand{\co}{\left[}
\newcommand{\cf}{\right]}
\newcommand{\cco}{\llbracket}
\newcommand{\ccf}{\rrbracket}
\newcommand{\R}{\mathbb R}

\newcommand{\N}{\mathbb N}

\newcommand{\dd}{\text{d}}

\newcommand{\na}{\nabla}
\newcommand{\1}{\mathbbm{1}} 

\newcommand{\rev}[1]{#1}

\usepackage{nicematrix}

\theoremstyle{thmstyleone}%
\newtheorem{theorem}{Theorem}
\newtheorem{proposition}[theorem]{Proposition}%

\theoremstyle{thmstyletwo}%
\newtheorem{remark}{Remark}%

\theoremstyle{thmstylethree}%

\begin{document}

\title[General-purpose post-sampling reweighting method for multimodal target measures]{General-purpose post-sampling reweighting method for multimodal target measures}


\author*[1,2]{\fnm{Pierre} \sur{Monmarch{\'e}}}\email{pierre.monmarche@univ-eiffel.fr}

\affil*[1]{\orgdiv{LAMA}, \orgname{Universit{\'e} Gustave Eiffel}, \orgaddress{\city{Champs-sur-Marne}, \postcode{77420},  \country{France}}}

\affil[2]{\orgname{Institut Universitaire de France}, \orgaddress{\country{France}}}


\abstract{When sampling multi-modal probability distributions, correctly estimating  the relative probability of each mode, even when the modes have been discovered and locally sampled, remains challenging. We test a simple reweighting scheme designed for this situation, which consists in minimizing (in terms of weights) the Kullback-Leibler divergence of a weighted (regularized) empirical distribution of the samples with respect to the target measure. \rev{We show  in a theoretical toy model that the convergence in terms of the number of samples follows the Monte Carlo rate. Numerical experiments on a variety of simple benchmarks are provided, exploring the effect of the ambient dimension, number of modes, number of samples and definition of the modes. We find that, particularly in the situation of interest with well-identified separated modes, the method is able to correct the error in the estimation from the initial sampling  of the relative probabilities of the modes. }
}




\maketitle

\section{Introduction}\label{sec:intro}

\paragraph{Sampling multimodal distributions.} Consider the problem of sampling a probability measure $\mu$ on $\R^d$ with density proportional to $e^{-U}$ with a known $U:\R^d \rightarrow \R$. In many applications  of statistical physics or Bayesian statistics (e.g. molecular simulations, model selection\dots)  the target $\mu$ is multi-modal, namely $U$ has several local minima, and  some domains with relatively large probability are separated by areas of much lower probability.
This makes the task difficult for the most standard approach, which is Markov Chain Monte Carlo (MCMC). Indeed, in this situation, schematically, we ask the MCMC sampler to solve three problems:
\begin{enumerate}[(i)]
\item Find the modes.
\item Sample locally each mode.
\item Estimate the relative weights between the different modes.
\end{enumerate}
However, usually, practical MCMC samplers are only efficient for the second point. Indeed, in most cases of interest, the problem is high-dimensional but the mass of $\mu$ is concentrated in the vicinity of some manifold of dimension much smaller than $d$, which makes large anisotropic jumps (uninformed of the global geometry of $\mu$) unpracticable. This only leaves MCMC samplers with local moves, either discretizations of continuous trajectories (Hamiltonian dynamics or diffusion processes) or chains with small jumps (e.g. Gaussian steps with a vanishing variance as $d$ increases), using local informations such as $U$ and $\na U$ evaluated at the current point or neighboring states. In the presence of multimodality, these local samplers exhibit a metastable behavior, meaning that transitions between modes is a rare event while the chain remains trapped for very long times in each mode it visits. As a consequence, the exploration of the space and thus the discovery of new modes is extremely slow. Worse, the correct equilibration between the modes requires to observe a lot of transitions, which happens at a time-scale out of reach in many practical situations.

Many variants of MCMC have been introduced to address this issue, based, among other ideas, on (adaptive) importance sampling (aimed at sampling a less metastable modification of the target) \cite{invernizzi2022exploration,Comer} or annealed or sequential sampling (sampling a path of probability measures, starting from a simple one and ending with the target) \cite{del2006sequential,annealing,chehab2024provable}. In adaptive methods, along sampling, some global knowledge on $\mu$ is learned (e.g. good collective variables \cite{belkacemi2021chasing,trizio2024advanced,ribeiro2018reweighted}, associated free energies \cite{Comer,M66}, scores or similar optimal drifts \cite{albergo2024nets}, densities~\cite{Gabrie2021a}). This often helps for solving both points (i) and (iii) (although it is not completely clear, for instance, that a good collective variable to describe a given mode is a good variable to distinguish the different modes).  We refer to \cite{Gabrie} for a review of many of these methods, and comparison on a common benchmark.

 In the present work, only point (iii) is addressed. This is based on the idea that, in many situations, it is in fact better to deal separately with the three parts of the problem. Due to dimensional complexity, solving point (i) in general is an extremely hard problem. In practice, each application has its model-dependent heuristics, exploiting additional structure or a-priori informations on the potential landscape. In any cases, when trying to solve (i) without dealing with (ii) and (iii), we do not have to keep track of the measure which is sampled by the method (contrary to e.g. importance sampling), so that we can use random initialisations with gradient descent (as is standard for non-convex optimization in machine learning), or completely non-equilibrium methods (e.g. saddle point search \cite{lelievre2024using,journel2023switched}) or possibly adaptive methods (e.g. OPES~\cite{invernizzi2022exploration}, or Gaussian-accelerated simulations~\cite{celerse2022efficient}) with very ``aggressive" parameters (which would yield a poor sampling due to high weight variance). In some situations, \rev{which actually correspond to applications we have particularly in mind here,} the modes may even be known a priori (for instance in protein-ligand binding in molecular simulations~\cite{ebrahimi2022symmetry}). In this paper we assume that step (i) has been addressed by some methods and that the modes are given. If some modes have been missed at this point, the method presented below will not find them. Moreover, we have in mind a situation where there is no more than a few dozens of modes (corresponding e.g. to relevant macroscopic states in molecular simulations, each of which  gathering possibly several close small-scale local energy minima).

As already mentioned, local MCMC are usually quite efficient for step (ii), so we assume now that we have samples in each mode. For instance, once the modes have been identified, we may simply run independent MCMC samplers initialized at each of these modes. Our goal is to perform step (iii) in this situation. Notice that, even without distinguishing the three steps, the question is also of interest in many cases with current practical methods (for instance adaptive importance sampling) because equilibration between modes happen at a time-scale larger than mode exploration and local sampling, which means that in practice the simulations often end with many modes being locally sampled but an incorrect relative weights between them. In the vocabulary of statistical physics, step (iii) corresponds to estimating free energy differences. Standard methods such a Bennet acceptance ratio~\cite{jia2020normalizing,bennett1976efficient} or targeted free energy \cite{jarzynski2002targeted,wirnsberger2020targeted} require to have some overlap between the samples in each state or to have some transformation between them, which is not easily obtained when the states are very separated and no intermediate samples in the separating region are available.

\paragraph{Our approach.} In summary, motivated by the above three-step viewpoint to multimodal sampling, our goal is to introduce a very simple general method for step (iii) when the modes have already been found and sampled. Contrary to most approaches we are aware of (for instance those reviewed in \cite{Gabrie}, or discussed in \cite[Appendix A]{lange2022interpolating}), the method does not produce any new sample (or doesn't add any new component to a mixture of parametrized distributions):  this is  a purely post-sampling step, which can be performed at the end of any sampling procedure.  In other words, it addresses the following question: when our sampling budget is all spent, what is the best we can do with its output ?

The method is a variational inference (VI) scheme, in the sense that it aims at minimizing the (reverse) Kullback-Leibler (KL) divergence (aka relative entropy) $\rho \mapsto \mathrm{KL}(\rho|\mu) = \int_{\R^d} \ln(\rho/\mu)\rho $ over a family of probability measures of the form $\rho = \sum_{k=1}^K p_k \nu_k$ with weights $p$ in the simplex $\Delta_K=\{p\in [0,1]^K,\ \sum_{k=1}^K p_k = 1\}$. In standard VI approaches, the distributions $\nu_k$ would also be in some parametric family of probability measures (typically, Gaussian, or in more general exponential families \cite{jerfel2021variational,minka2001family,campbell2019universal}). The specificity of the method in the present work is that, instead, in a somewhat semi-parametric way, we use for $\nu_k$ the empirical distribution of the local samples in mode $k$ (up to regularization so that the KL makes sense). The idea of combining MCMC sampling with VI goes back at least to   \cite{DeFreitas} and has been derived in a variety of schemes since then (as further discussed below) but we haven't been able to find previous references on the method studied here. It should be emphasized that  considering that the modes are known is obviously  a very strong assumption, while most of the existing enhanced/adaptive sampling or VI methods for multi-modal sampling are designed to address (i), (ii) and (iii) simultaneously. However, as discussed above, the specific question of solving (iii) alone is also of interest. The good results of the method obtained in our numerical experiments hence motivated the present report.

\paragraph{Organization of the work.}  Related methods are discussed at the end of the introduction. The method is described in Section~\ref{sec:themethod}. Section~\ref{sec:theory} provides a brief theoretical discussion, while numerical experiments are gathered in Section~\ref{sec:numerique}.

\paragraph{Relation with other methods.} Since~\cite{DeFreitas}, a variety of methods have combined VI and sampling. \rev{In one direction, MCMC can be used to provide a richer family of probability measures for VI, e.g. by propagating  through a few MCMC steps a parametrized measure~\cite{ruiz2019contrastive,salimans2015markov}, or by sampling latent variables from which the conditional densities of the observations are parametrized~\cite{hoffman2017learning}. Our method can be seen like this, in an extreme way since the component of the mixture $\pi(p)$ are just directly given by sampling. In the other direction, VI can be used to design (with mixture of Gaussian distributions, normalizing flows, etc.) good proposal kernels for importance sampling or MCMC \cite{Gabrie2021a,zimmermann2021nested,domke2018importance}, see also references therein.} In this last situation, and similarly in adaptive biasing importance sampling methods \cite{invernizzi2022exploration,Comer}, the idea is that, once the modes have been found, the adaptive mechanism will converge to a stationary Markovian dynamics (or an importance sampling proposal for i.i.d. samples) which performs transitions between modes easily, either thanks to a proposal which allows for jumps between modes, or because some high-probability paths have been created between them. This means that, once this idealized sampler has been reached, we still have to wait for the Monte-Carlo estimation of the relative weights. By contrast, the point of view taken in the present work is that, since the local MCMC samples are good, we don't expect to get more information by producing more samples and we estimate deterministically the weights.

Another  method related to ours is the confinment approach in \cite{cecchini2009calculation}, which also attempts at estimating the relative weights without having to observe transitions between the states, by progressively transforming each mode into a Gaussian distribution, estimating the free energy difference along this transformation, and then using that the absolute free energy is explicit in the harmonic case.

In \cite{molina2024active}, densities are estimated in each mode, and then a mixture of these densities is used in a Metropolis-Hastings MCMC procedure. Contrary to our method, the weights of the mixture are adaptively estimated along the MCMC sampling.

\section{The method}\label{sec:themethod}

\subsection{Prerequisite} \label{sec:prerequis}

The whole algorithm described in the present work gathers several parts. Each of these steps have to be efficiently addressed in practical cases, which, it should be emphasized, may raise difficulties. However, the focus of this work is on the specific optimization problem in the space of probability measure that is used. The other ingredients are by themselves the topics of specific fields of research, to which we do not claim any contribution here. We describe them in the following.

\begin{enumerate}
\item \textbf{Input data.} We assume that we have $N$ samples $(x_i)_{i\in\cco 1,N\ccf} \in (\R^d)^N$, and that the energies $(U(x_i))_{i\in\cco 1,N\ccf}$ have been computed. As discussed in Section~\ref{sec:intro}, we typically think of these samples as having been obtained either by first exploring the energy landscape to find the modes of $\mu \propto e^{-U}$ and then by using standard MCMC samplers in each mode, or by a biased MCMC sampler. In the rest of the method, no other sample will be generated. In particular, as already mentioned, if a mode is not represented in the samples, it will obviously be missed by the method.

\item \textbf{Clustering.} We assume that the samples are sorted in $K \in \cco 2,N\ccf $ clusters, i.e. that we have a partition $\cco 1,N\ccf = \sqcup_{k=1}^K I_k$ so that $(x_i)_{i\in I_k}$ is the $k^{th}$ cluster. Write $n_k=\mathrm{Card}(I_k)$.  There may be several methods to define the clusters. If the samples have been generated by local MCMC samplers initialized at different local minimizers of $U$ \rev{(corresponding to step (ii) of the global approach discussed in the introduction)}, then all the points generated starting from a given minimizer can be gathered in the same cluster. Alternatively, if we are given the $N$ samples without any indication, we may use any general-purpose clustering algorithm \cite{ester1996density,ezugwu2022comprehensive,westerlund2019inflecs} which, in the absence of further information, may also be used to determine $K$. The clustering may also come from a partition along the value of some reaction coordinates (as in Figure~\ref{fig-Langevin} below). Avoiding this clustering step by chosing $K=N$ is also possible (see  numerical experiment in Section~\ref{subsec:expweight}) but for now and in most of the work we have in mind the case where $K \ll N$, and specifically $K \ll \min\{n_k,\ 1\leqslant k \leqslant K\}$.

\item \textbf{Density estimation.} We assume that in each cluster $k\in\cco 1,K\ccf$, we have an estimate $\nu_k(x)$ of the density of $(x_i)_{i\in I_k}$, i.e. a function $\nu_k:\R^d \rightarrow \R_+$ which can be evaluated in practice and should be such that, in some suitable sense,
\begin{equation}
\label{eq:estim-densite}
\frac1{n_k} \sum_{j\in I_k} \delta_{x_j} \ \simeq\ \nu_k(x) \dd x\,.
\end{equation}
There is a large literature on this central problem of data analysis, see \cite{wang2019nonparametric,biroli2026kernel}. Notice that the clustering step is expected to lead the density in each cluster to be relatively simple (e.g. essentially unimodal), avoiding or at least mitigating difficulties such as mode collapse \cite{soletskyi2025theoretical}. A difficulty is that $d$ is typically large. In many high-dimensional applications, an important point is to find some features (also called collective variables or reaction coordinates in the field of molecular simulations) $\xi:\R^d\rightarrow \R^m$ which are low-dimensional (or at least moderate-dimensional, e.g. $m \simeq 100$, in any case small with respect to $d$) and such that $\xi(x)$ contains much of the information of the sample $x$ \cite{belkacemi2021chasing,trizio2024advanced,doi:10.1021/acs.jpcb.3c07075}. Typically, this means that the conditional distribution of $x$ given $\xi(x)$, under $\mu$ or under the clustered densities $\nu_k$, should be simple. As a consequence, a natural way to perform the density estimation in $\R^d$ is to find a reparametrization (i.e. a $\mathcal C^1$ diffeomorphism) of the form $\varphi(x) = (\xi(x),\xi^{\perp}(x))$, estimate the marginal density of $\xi(x)$ by standard methods in moderate dimension (e.g. Gaussian kernel estimation), then the conditional density of $\xi^{\perp}(x)$  given $\xi(x)$ by a Gaussian density (see \cite{dutordoir2018gaussian} and references therein for variations), and then obtain the density $\nu_k$ by the change of variable formula (see the examples in Section~\ref{sec:numerique}).
\end{enumerate}

\subsection{The optimization problem}

We use the notations introduced in the previous section.

\paragraph{Parametrization.} For $p$ in the simplex $\Delta_K = \{p\in[0,1]^K,\sum_{i=1}^K p_i = 1\}$, we write 
\begin{align}
\label{eq:defpi(p)}
\pi(p) &= \sum_{k=1}^K  p_k \nu_k \,.
\end{align}
Our goal is to determine $p\in\Delta_K$ so that $\pi(p)$ is in a suitable sense the best approximation of $\mu$ that we can get simply by reweighting the samples between different clusters. Denoting by $p^*\in\Delta_K$ the parameter obtained at the end of the algorithm, we will  thus have, first, in a VI spirit, $\pi(p^*)$ as a computable density approximation of $\mu$ and, second, alternatively, thanks to~\eqref{eq:estim-densite}, the possibility to estimate the  expectation of an observable $f\in L^1(\mu)$ with respect to $\mu$ with the reweighted Monte Carlo estimator
\begin{equation}
\label{eq:weightedestimator}
\int_{\R^d} f(x)\mu( \dd x) \simeq \sum_{k=1}^K \frac{p_k^*}{n_k}  \sum_{j\in I_k} f(x_j)\,.
\end{equation}

\paragraph{Objective function.} We would like to determine $p^*$ by minimizing
\begin{align*}
p \mapsto J(p) := \mathrm{KL}\po \pi(p)|\mu\pf &= \int_{\R^d}  \co U(x) + \ln\pi(p)(x)\cf \pi(p)(x)\dd x  + \mathrm{constant} \,.
\end{align*}
This is particularly  consistent if the local samples have been generated with an overdamped Langevin scheme, since the latter is the Wasserstein gradient flow of $\rho \mapsto \mathrm{KL}(\rho|\mu)$. In this situation, both the local samplers and the reweighting step have the same objective. Moreover, if they have been obtained by an MCMC sampler targetting $\mu$ and initialized in a mode of $\mu$, the $\nu_k$'s should be concentrated in regions of the space that have a non-negligible mass with respect to $\mu$, making $\mathrm{KL}\po \pi(p)|\mu\pf$ a well-behaved quantity (by contrast with $\mathrm{KL}\po \mu| \pi(p)\pf$). A known issue when minimizing  $\mathrm{KL}\po \cdot |\mu\pf$ is that the tail of the distribution is underestimated \cite{jerfel2021variational}, but we are focusing on localized data (most molecular simulations for instance are in a compact set with boundary conditions).

\paragraph{Optimization.} Since $J$ and the simplex are convex, the optimization problem is nice. Among other possibilities, in this work we use the projected exponential gradient descent as in \cite{kivinen1997exponentiated}, given by
\begin{equation}\label{eq1}
p_{k}^{m+1} = \frac{p_k^m \exp\po -\delta \partial_{p_k} J(p^{m})\pf}{\sum_{\ell=1}^K p_\ell^m \exp\po -\delta \partial_{p_\ell} J(p^{m})\pf }\,,\qquad k\in\cco 1,K\ccf,
\end{equation}
for a step-size $\delta>0$, with the gradient of $J$ (seen as a function in $\R^d$; notice that~\eqref{eq1} is already taking into account the constraints on $p \in \Delta_K$) given by
\begin{equation}
\label{eq:nablaJp}
\partial_{p_k} J(p) = 1+ \int_{\R^d} \po U + \ln \pi(p) \pf \nu_k
\end{equation}
(we can ignore the additive constant $1$, which disappears in~\eqref{eq1}). \rev{Notice that, in~\cite{kivinen1997exponentiated},~\eqref{eq1} is motivated as an entropic mirror descent, which appears in other schemes like~\cite{korba2022adaptive}. The difference between~\cite{korba2022adaptive} and our case is that we apply it on the low-dimensional weights $p \in \Delta_K$ instead of the probability space $\mathcal P(\R^d)$.  }

The quantity~\eqref{eq:nablaJp} is not tractable so that, using~\eqref{eq:estim-densite}, we  use the stochastic gradient approximation 
\begin{equation}
\label{eq:gradient-sto}
\partial_{p_k} J(p) \simeq  1+ V_k +   \frac{1}{n_k} \sum_{j\in I_k}  \ln \pi(p)(x_j) 
= 1+ V_k +   \frac{1}{n_k} \sum_{j\in I_k}  \ln \po \sum_{\ell=1}^K p_\ell \nu_\ell(x_j)\pf 
\end{equation}
with
\[V_k = \frac{1}{n_k} \sum_{j\in I_k}  U(x_j)\,, \]
computed once at the beginning of the optimization procedure. If $n_k$ is large, at each gradient descent iteration, the last sum in \eqref{eq:gradient-sto} can be further approximated by subsampling.

\begin{remark}[Case without overlap]\label{rem:approximation}  When the clusters are well separated (which may be the case in some applications as protein-ligand binding~\cite{ebrahimi2022symmetry} or if the samples have been re-grouped using a spatial clustering as in Figure~\ref{fig-Langevin}), in the last term of~\eqref{eq:gradient-sto}, for $j\in I_k$, $\sum_{\ell=1}^K p_\ell \nu_\ell(x_j) \simeq p_k \nu_{k}(x_j)$ (for a given $p\in\Delta_K$ with $p_k\neq 0$). Plugging this approximation in~\eqref{eq:gradient-sto} gives 
\begin{equation}
\label{eq:gradient-sto2}
\partial_{p_k} J(p) \simeq  1+ W_k +    \ln \po  p_k \pf 
\end{equation}
with
\[W_k = \frac{1}{n_k} \sum_{j\in I_k}  \co U(x_j) + \ln \nu_k(x_j)\cf\,. \]
Instead of running the gradient descent with this approximation, we can directly compute its limit, which is
\begin{equation}
\label{eq:p*no-overlap}
p_k^* = \frac{e^{-W_k}}{\sum_{\ell=1}^K e^{-W_\ell}}\,.
\end{equation}
This is the unique minimizer in $\Delta_K$ of 
\[p\mapsto \sum_{k=1}^K p_k \po \ln p_k + W_k\pf\,, \]
which is indeed an approximation of $J$ provided the density estimation~\eqref{eq:estim-densite} is good and there is no overlap between the modes. When the separation assumption on the modes is not reasonable, the weights~\eqref{eq:p*no-overlap} can serve as initialization for the gradient descent.

Notice that, if the densities $\nu_k$ were not obtained from the empirical distributions of local samples but from a parametrized family as in standard VI~\cite{jerfel2021variational,minka2001family,campbell2019universal,miller2017variational}, then this no-overlap situation would typically not be met since several component would have to be combined to give a good representation of each mode. 
\end{remark}

\begin{remark}[Dimension reduction]\label{rem:reduction}
Assume that the states  can be decomposed as $x =(y,z)\in\R^{n+m}$ where $y$ are the main variables and $z$ are some high-dimensional non-informative fluctuations. Assume further that the conditional density $\nu(z|y)$ of $z$ given $y$ is the same in all clusters, and is the Gaussian density $\mathcal N(m_z(y),\Sigma_z)$ for some function $m:\R^n\rightarrow\R^m$ and a fixed covariance matrix $\Sigma_z$. Then, for $p\in\Delta_K$ denoting by $\nu_k^y$ the marginal density of $y$ in cluster $k$ and by $\pi^y(p) = \sum_{k=1}^K p_k \nu_k^y$,
\[\int_{\R^d} \pi(p) \ln \pi(p)  = \int_{\R^n} \pi^y(p) \ln \pi^y(p) + \int_{\R^n} \pi^y(p)(y) \co \int_{\R^m} \nu(\cdot|y) \ln\nu(\cdot|y)\cf  \dd y   \,. \]
The entropy of a Gaussian distribution being independent from its mean,  $\int_{\R^m} \nu(\cdot|y) \ln\nu(\cdot|y)$ is in fact independent from $y$, so that the last term is independent from $p$. This means that, in this situation, when computing $J$, or more precisely its gradient, it is sufficient to estimate the marginal densities of the variables $y$. The algorithm can be applied based only on the data $(y_i,U(x_i))$, i.e. the variables $z_i$ can be discarded (making the density estimation much easier).

In more general situations this discussion can serve as a justification to replace the $d$-dimensional problem by a smaller $n$-dimensional problem (with, for instance, $n$ over the order $100$ as in Section~\ref{sec:numerique}) by discarding some variables, even if  the discarded variables have a conditional density with respect to the other variables which is not exactly the same in each cluster of which is not exactly Gaussian with a fixed variance, as long as this condition is approximately met.

\end{remark}

\section{Analysis in simple situations}\label{sec:theory}

\subsection{\rev{Case with Gaussian and stretched exponential tails}}

In order to get some insights on the \rev{stochastic gradient} approximation~\eqref{eq:gradient-sto} made during the derivation of the algorithm, in this section we  consider the following toy problem: assume that $\mu = \sum_{k=1}^K p_k^* \nu_k$ for some $p^*\in\Delta_K$ and some \emph{known} probability densities $\nu_k$, and that  for each mode $k\in \cco 1,K\ccf$  we are given $N$ i.i.d. samples $(X_{1,k},\dots,X_{N,k})$ distributed according to $\nu_k$. In particular, $p^*$ is clearly a global minimizer of \rev{the function $J:\Delta_K\to\R$ defined by}
\[J(p) = \mathrm{KL}(\pi(p)| \mu) := \mathrm{KL}(\pi(p)| \pi(p^*))\,,\]
with $\pi(p)$ given by the expression~\eqref{eq:defpi(p)} (although here this is with the true densities $\nu_k$ from which the data has been sampled, instead of an estimation as in the genuine algorithm). Besides, it is the  unique global minimizer provided the densities $\nu_k$ are linearly independent (in any cases, if this condition is not met, then $\pi(p)=\mu$ for all global minimizers).  Moreover, $J$ is  convex, and more precisely a direct computation shows that, for all $u\in\R^K$ and $p\in\Delta_K$,
\[\sum_{i,j=1}^K u_iu_j \partial_{p_i}\partial_{p_j} J(p) = \int_{\R^d} \frac{\po \sum_{i=1}^K \nu_i(y) u_i\pf^2 }{\sum_{i=1}^K p_i \nu_i(y)}\dd y\,.\]
In particular, if the $\nu_k$ are linearly independent, \rev{the integrand is positive on a set of positive measure, so that $u\cdot \na^2 J(p)u>0$ for all $p\in\Delta_K$ and $u\in\R^K$, namely} $J$ is strictly convex. In any cases, by convexity, the gradient descent $(p^m)_{m\in\N}$ given in~\eqref{eq1}  converges quickly to the global minimizers\rev{, see \cite[Theorem 5.1]{beck2003mirror}}.

However, due to the approximation~\eqref{eq:gradient-sto}, what we have is $(q^m)_{m\in\N}$ given by
\begin{equation}\label{eq:descente-approx}
q_{k}^{m+1} = \frac{q_k^m \exp\po -\delta G_k(q^{m})\pf}{\sum_{\ell=1}^K q_\ell^m \exp\po -\delta G_\ell(q^{m})\pf }\,,\qquad k\in\cco 1,K\ccf,
\end{equation}
with, denoting $U = -\ln \pi(p^*)$,
\begin{equation}
\label{eq:locGp}
G_k(p) = \frac1N \sum_{i=1}^N U(X_{i,k})  +   \frac{1}{N} \sum_{i=1}^N  \ln \po \sum_{\ell=1}^K p_\ell \nu_\ell(X_{i,k})\pf \,. 
\end{equation}
In the present framework, this is an unbiased estimator of $\na J(p)$, i.e. for all $p\in\Delta_K$,
\begin{equation}
\label{loc:sto}
\mathbb E \po G_k(p)\pf = \partial_{p_k} J(p)\,. 
\end{equation}
However, contrary to the more classical settings of stochastic gradient descent, the Monte Carlo estimation of the gradient is not renewed at each iterations: the samples $X_{i,k}$ are the same over the whole trajectory. In particular, $q^m$ is random and depends on the samples, and thus, conditionally to $q^m$, $G_k(q^m)$ is not an unbiased estimator of $\partial_{p_k} J(q^m)$.

The next result shows that, as a function of $N$, the error due to this approximation still follows the classical Monte Carlo rate.

\begin{proposition}\label{prop:main}
Assume that there exist $C,a>0$ and $\rev{\alpha\in(0,2]}$  such that for all $x\in\R^d$ and $k\in\cco 1,K\ccf$,
\begin{equation}
\label{eq:hypothese}
C^{-1} e^{-a|x|^{\rev{\alpha}}} \leqslant \nu_k(x) \leqslant C e^{-a|x|^{\rev{\alpha}}}\,.
\end{equation}
Then there exists $L>0$ such that, for any initial condition $r^0\in\Delta_K$, considering $(p^m)_{m\in\N}$ and $(q^m)_{m\in\N}$ respectively given by~\eqref{eq1} and \eqref{eq:descente-approx} with $p^0 = q^0=r^0$, for all $m\in\N$,
\begin{equation}
\label{eq:prop}
\mathbb E \po |p^m- q^m |^2\pf \leqslant L e^{L \delta m}\po N^{-1} + \delta^2\pf \,.
\end{equation}
\end{proposition}

Since the limit problem (taking $N\rightarrow \infty$) is globally converging (along the descent iterations $m$), we may expect a similar result to hold uniformly in time, i.e.~\eqref{eq:prop} with a right hand side independent from $m$\rev{, as classically happens in mean-field optimisation problems~\cite{Chewietal}}. In practice this doesn't change much since we only perform short trajectories.

\begin{proof}
\rev{Under the condition~\eqref{eq:hypothese}, there exists $C'>0$ such that for any $p\in\Delta_K$ and $x\in\R^d$, 
\begin{equation}
\label{eq:lnpix}
|\ln \pi(p)(x)|  \leqslant C'(1+ |x|^2)\,.  
\end{equation}
In particular this holds for $\ln \pi(p_*) = -U$. Integrating this over $\nu_k$, using again~\eqref{eq:hypothese} and recalling the expression~\eqref{eq:nablaJp} for $\na J$, this gives}
\begin{equation}
\label{eq:locsuppF1}
\sup_{p\in\Delta_K}\max_{k\in\cco 1,K\ccf} |\partial_{p_k} J(p) | =: M_0 <\infty\,. 
\end{equation}
\rev{Denote
\[g_k(p,x) = \frac1N \sum_{i=1}^N U(x_{i,k})  +   \frac{1}{N} \sum_{i=1}^N  \ln \po \sum_{\ell=1}^K p_\ell \nu_\ell(x_{i,k})\pf  \]
for any $x=(x_{i,k})_{i\in\cco 1,N\ccf,k\in\cco 1,K\ccf} \in (\R^{d})^{KN}$.  Thanks to~\eqref{eq:lnpix},
\begin{equation}
\label{eq:locsuppG1}
\sup_{p\in\Delta_K}\max_{k\in\cco 1,K\ccf}|g_k(p,x)| \leqslant 2C'\po 1+ \max_{k\in\cco 1,K\ccf} \frac{1}{N} \sum_{j=1}^N |x_{j,k}|^2\pf\,.  
\end{equation}
Moreover, for any $y\in\R^d$, using~\eqref{eq:hypothese},
\[\partial_{p_k} \ln \pi(p)(y)  = \frac{\nu_k(y)}{\sum_{j=1}^Kp_j \nu_j(y)} \leqslant C^2\,, \]
from which, for all $x\in(\R^d)^{KN}$,
\begin{equation}
\label{eq:locsuppG2}
\sup_{p,q\in\Delta_K}\max_{k\in\cco 1,K\ccf}|g_k(p,x)-g_k(q,x)|\leqslant C^2|q-p|\,.
\end{equation}
Similarly, in view of~\eqref{eq:nablaJp}, we get that
\begin{equation}
\label{eq:locsuppF2}
\sup_{p,q\in\Delta_K}\max_{k\in\cco 1,K\ccf}|J(p)-J(q)|\leqslant C^2|q-p|\,.
\end{equation}}

Moreover, \rev{by the concentration properties of $\alpha$-sub-exponential distributions (for instance \cite[Proposition 1.1]{10.1214/21-EJP606})} we can find $M_1>0$ large enough so that the event
\[\mathcal A= \left\{ \max_{k\in\cco 1,K\ccf} \frac{1}{N} \sum_{j=1}^N |X_{j,k}|^2 \leqslant M_1 \right\}\]
has a probability larger than $1- \exp(-N^{\rev{\alpha/2}})$.  

\rev{Since the definition~\eqref{eq:locGp} reads $G_k(p)=g_k(p,X)$ (with the random samples $X$)}, in view of~\eqref{eq:locsuppG1} and~\eqref{eq:locsuppG2} then, almost surely,
\begin{equation}
\label{boundG2}
\sup_{p,q\in\Delta_K}\max_{k\in\cco 1,K\ccf}|G_k(p)-G_k(q)|\leqslant C^2|q-p|
\end{equation}
and
\begin{equation}
\label{boundG}
\sup_{p\in\Delta_K}\max_{k\in\cco 1,K\ccf}|G_k(p)|\1_{\mathcal A}\leqslant \rev{2C'(1+M_1)\,.}
\end{equation}
\rev{Now, for $u\in \R^K$ and $p\in\Delta_K$, a Taylor expansion of $f_k(\delta) := p_k e^{-\delta u_k }/(\sum_{j=1}^K p_j e^{-\delta u_j})$ shows that
\[\left | f_k(\delta) - p_k - \delta f_k'(0)\right| \leqslant M_1(u) \delta^2\]
for some constant $M_1(u)>0$ which is independent from $p\in\Delta_K,k\in\cco 1,K\ccf$ and $\delta\in[0,1]$ and can be bounded uniquely in terms of $|u|$. From~\eqref{eq:locsuppF1} and~\eqref{eq:locsuppG1} we get that, for some $M_2>0$, for all $m\in\N$,
\begin{align}
\left |p^{m+1}_k - p_k^m + \delta  p_k^m\co \partial_k J(p^m) - \sum_{\ell=1}^K p_\ell^m \partial_\ell J(p^m)\cf \right|  &\leqslant M_2 \delta^2 \label{loc:pm1} \\
\left |q^{m+1}_k - q_k^m + \delta  q_k^m\co G_k(q^m) - \sum_{\ell=1}^K q_\ell^m G_\ell(q^m)\cf \right|\1_{\mathcal A} &\leqslant M_2 \delta^2 \,. \label{loc:pm2} 
\end{align}
}

For $m\in\N$, \rev{using that $|p^{m+1}- q^{m+1} | \leqslant \mathrm{diam}(\Delta_K)=\sqrt{2}$,} we simply bound
\[ \mathbb E \po |p^{m+1}- q^{m+1} |^2\pf \leqslant \mathbb E \po |p^{m+1}- q^{m+1} |^2 \1_{\mathcal A}\pf + \rev{2 \mathbb P(\mathcal A^c)}\,. \]
\rev{Since $\mathbb P(\mathcal A^c) \leqslant \exp(-N^{\rev{\alpha/2}})$, the second term is} negligible with respect to the right hand side of~\eqref{eq:prop}, we may simply ignore it. To treat the first term, \rev{we use the expansions~\eqref{loc:pm1}  and~\eqref{loc:pm2} and then the bounds~\eqref{eq:locsuppF1}, \eqref{eq:locsuppG1}, \eqref{eq:locsuppG2} and~\eqref{eq:locsuppF2}  to get
\begin{align*}
|p^{m+1}- q^{m+1} |\1_{\mathcal A} & \leqslant (1+\delta M_3) |p^m - q^m|\1_{\mathcal A} + 2M_2 \delta^2 + \delta M_3  \sum_{j=1}^K  |\partial_{p_j} J(p^m)-G_j(p^m)| \,,
\end{align*}
for some $M_3>0$. Taking the square and then the expectation,  and using that $(a+\delta b)^2 \leqslant (1+\delta)a^2 + (\delta+\delta^2) b^2$ for any $a,b\in\R$, we get that, for some $M_4>0$, 
} 
\begin{eqnarray}
\lefteqn{\mathbb E \po |p^{m+1}- q^{m+1} |^2 \1_{\mathcal A}\pf} \nonumber\\
& \leqslant &  (1+M_4\delta) \mathbb E \po |p^{m}- q^{m} |^2 \1_{\mathcal A}\pf  + M_4 \delta\sum_{k=1}^K \mathbb E \po |\partial_{p_k} J(p^m)-G_k(p^m)|^2\pf + M_4 \delta^3\,.\label{loc:sdfsdgd}
\end{eqnarray}
 Finally, since $p^m$ is deterministic, we can use~\eqref{loc:sto} and that, under~\eqref{eq:hypothese}, for eack $k\in\cco 1,K\ccf$, the random variable $U(X_{1,k})  +     \ln \po \sum_{\ell=1}^K p_\ell \nu_\ell(X_{1,k})\pf $ is bounded from above and below by a quadratic function independent from $p$, hence has a variance bounded uniformly in $p$, so that 
 \begin{equation}
 \label{loc;dfsdgsd}
\sum_{k=1}^K \mathbb E \po |\partial_{p_k} J(p^m)-G_k(p^m)|^2\pf \leqslant \frac{M_5}{N}
 \end{equation}
for some $M_5>0$. Plugging this in~\eqref{loc:sdfsdgd}, the conclusion follows from Gr{\"o}nwall's lemma.
\end{proof}

\begin{remark}
\rev{In Proposition~\ref{prop:main} we haven't taken into account the error due to the density estimation, since we assumed that in~\eqref{eq:locGp} the $\nu_k$'s are the true density of the samples. However, we can extend as follows the arguments above to the situation were, given $N$ i.i.d. samples $(X_{i,k})_{i\in\cco 1,N\ccf}$ with density $\nu_k$, we replace this true density in~\eqref{eq:locGp} by an estimation $\tilde \nu_k$ based on the samples. Assuming that, for some constants $C,a,\alpha$, a bound of the form~\eqref{eq:hypothese} also holds for $\tilde \nu_k$ with high probability, we can follow the previous computation and straightforwardly obtain~\eqref{loc:sdfsdgd}, where we have added in the definition of the event $\mathcal A$ the fact that the bound~\eqref{eq:hypothese} holds for the $\tilde \nu_k$'s. Introducing $\hat G_k$ given by~\eqref{eq:locGp} with the true densities $\nu_k$, by a triangular inequality and reasoning as in~\eqref{loc;dfsdgsd} we get
\[ \sum_{k=1}^K \mathbb E \po |\partial_{p_k} J(p^m)-G_k(p^m)|^2\pf \leqslant 2\frac{M_5}{N} +  2  \sum_{k=1}^K \mathbb E \po |\hat G_k(p^m)-G_k(p^m)|^2\pf
 \,.\] 
 By the Cauchy-Schwarz inequality, we can bound the last term as follows: 
 \begin{equation}
 \label{loc:sqdfsgpazer}
\mathbb E \po |\hat G_k(p^m)-G_k(p^m)|^2 \pf \leqslant \int_{\R^d} \left|\ln \sum_{\ell=1}^K p_k \nu_\ell(x) - \ln \sum_{\ell=1}^K p_k \tilde \nu_\ell(x)  \right|^2 \nu_k(x)\dd x\,.
 \end{equation}
 As a conclusion, we get a result similar to Proposition~\ref{prop:main}, with two additional terms in the right hand side of~\eqref{eq:prop}: first, twice the probability that the bound~\eqref{eq:hypothese} does not hold and, second, the supremum over $p\in\Delta_K$ of the log-density error in the right hand side of~\eqref{loc:sqdfsgpazer}. The classical study of such log-density errors usually assumes that the density $\nu_k$ is log-concave,  see \cite{chen2024new} and references therein. Indeed, contrary to more standard $L^p$ error analysis on the density estimation, it is necessary to control regions with low probability density. We stated Proposition~\ref{prop:main} in the form above to show that the error in terms of $N$ due to the stochastic gradient error~\eqref{eq:gradient-sto} can be controlled without further conditions than the tail control~\eqref{eq:hypothese}.}
 
\rev{Notice that, beyond the density estimation error, Proposition~\ref{prop:main} also ignores the error possibly induced by the definition of the clusters, which may merge or forget modes. The study of this contribution to the error depend on the specific method used to defined the modes, and seems less straightforward to control in our situation relying on existing results. See Section~\ref{sec:multiwellLangevin} for numerical experiments on this topic.}
\end{remark}

\subsection{Case with disjoint supports}\label{subsec:separated}

Contrary to the previous section, here we consider the genuine algorithm without assuming that the samples are i.i.d. distributed according to the densities $\nu_k$.

In fact, due to our motivation to sample separate modes, we will see that the experiments in Section~\ref{sec:numerique} do not satisfy~\eqref{eq:hypothese}, and are better described by an even simpler condition: the estimated densities $\nu_k$ all have disjoint support.

Assume that this holds and denote by $\mathcal D_k$ the support of $\nu_k$ and
\[\mu_k(x) = \frac{\mu(x) \1_{\mathcal D_k}(x)}{\mu(\mathcal D_k)}\,,\qquad q^*_k= \frac{\mu(\mathcal D_k)}{\sum_{\ell=1}^K \mu(\mathcal D_\ell)}\,,\qquad r= 1 - \sum_{\ell=1}^K \mu(\mathcal D_\ell) \]
so that $\mu_{\mathcal D}:=\sum_{k=1}^K q_k^* \mu_k$ is  the measure $\mu$ conditioned on $\mathcal D= \cup_{k=1}^K \mathcal D_k$.  Even when the missing mass $r$ is positive, determining $q_k^*$ is enough to get ratios of probabilities (free energy differences) since
\[\frac{q_k^*}{q_\ell^*} = \frac{\mu(\mathcal D_k)}{\mu(\mathcal D_\ell)}\]
(in this disjoint support case, we retrieve as discussed e.g. in \cite[Section 2]{Gabrie} the link between the mode weight estimation and the computation of the normalization constant $\mu(\mathcal D_k)$ of the measure $\1_{\mathcal D_k} \mu$). More generally the goal of the algorithm in this situation is to find the best approximation of $\mu_{\mathcal D}$.

In the present settings,  the approximation discussed in Remark~\ref{rem:approximation} is in fact exact and thus the algorithm outputs the weights proportional to $p_k^*\propto e^{-W_k}$ with
\begin{equation}
\label{locWk}
W_k = \frac{1}{n_k} \sum_{j\in I_k}  \co U(x_j) + \ln \nu_k(x_j)\cf  = \frac{1}{n_k} \sum_{j\in I_k} \ln \po \frac{\nu_k(x_j)}{q_k^*\mu_k(x_j)}\pf - \ln(1-r)\,.
\end{equation}
  In other words,
  \begin{equation}
  \label{eqloc:p*}
  p_k^* \propto  q_k^* \exp\co \frac{1}{n_k} \sum_{j\in I_k} \ln \po \frac{\mu_k(x_j)}{\nu_k(x_j)}\pf  \cf\,.
  \end{equation}
  Assuming that the condition~\eqref{eq:estim-densite} is satisfied as $n_k\rightarrow \infty$ (which is true for instance if the samples in cluster $k$ are i.i.d. with law $\nu_k$, or are given by an ergodic MCMC sampler with target measure $\nu_k$), then 
  \begin{equation}
\label{loceq}  
  p_k^* \underset{\min_{\ell} n_\ell \rightarrow \infty}{\longrightarrow} \bar p_k^* \propto  q_k^* \exp\co - \mathrm{KL}(\nu_k|\mu_k)  \cf\,, 
  \end{equation}
which is the minimizer of 
\begin{equation}
\label{locJp}
J(p) = \sum_{k=1}^K p_k\int_{\mathcal D_k} \po U + \ln (p_k\nu_k)\pf  \nu_k\,. 
\end{equation}
 If $\nu_k = \mu_k$, the correct weights are recovered. However in general if $\nu_k\neq \mu_k$ there is still an error, which is expected since within a cluster the different samples are not re-weighted one relatively to the others. It is interesting to see that this asymptotic error is minimized by minimizing $\mathrm{KL}(\nu_k|\mu_k) $ (which is in fact simply a consequence of the extensivity property of the relative entropy), which is precisely the quantity that the algorithm aims at minimizing at the global level. This means that, in multi-scale situations, it is consistent to apply the method hierarchically, i.e. combining small clusters together with re-weighting in small groups, and then merging some of these groups with a new relative re-weighting, etc. According to~\eqref{loceq}, taking the optimal weights at the smaller scale minimizes the error at the larger scales. 
 
\medskip 
 
 \rev{Let ut now compare the output $p^*$ in~\eqref{eqloc:p*} with the result of an importance sampling approach. Using that
\[
 q_k^* = \int_{\R^d} \1_{\mathcal D_k} \mu(x) \dd x \propto \int_{\R^d}   \omega_k(x) \nu_k(x) \dd x \,,\qquad \omega_k(x) = \frac{e^{-U(x)}}{\nu_k(x)}\,,
\]
 a natural importance sampling estimator relying on the same ingredients  as our approach  (in particular, the estimation of $\nu_k$ in each cluster) is
  \begin{equation}
 \label{eqloc:q*}
 \tilde p_k^* \propto \frac{1}{n_k} \sum_{j\in I_k}   \frac{e^{-U(x_j)}}{\nu_k(x_j)}\,.
 \end{equation}
 In particular, assuming as before \eqref{eq:estim-densite} when $n_k\rightarrow \infty$, and that $\nu_k$ also converges to some limit (so that $\omega_k$ converges), we get that $\tilde p^* \rightarrow q^*$ (whether $\nu_k =\mu_k$ or not), to be compared with~\eqref{loceq}. The two methods do not have the same goal: the importance sampling estimator~\eqref{eqloc:q*} is designed to converge to $q^*$, while~\eqref{eqloc:p*}  aims at finding the optimal representation (in terms of the $\mathrm{KL}$ divergence) of $\mu$ as a mixture of the $\nu_k$'s. In view of~\eqref{loceq}, these two questions have the same answer $q^*$ only if $\nu_k = \mu_k$ for $k\in\cco 1,K\ccf$. The interesting thing to notice here is that, even when this condition is met, the two methods have different outputs: from~\eqref{locWk}, we see that $W_k=-\ln p_k^*$ is a Monte Carlo estimation of $-\int_{\R^d}\ln \omega_k \nu_k$. Since the densities $\nu_k$ are designed to be close to restrictions of $\mu$ on some domains $\mathcal D_k$,  the weights $\omega_k$ are typically bounded from below by a positive constant, however they may take very high values, especially at the boundary of $\mathcal D_k$ where an estimated density $\nu_k$ may vanish. Hence, taking the logarithm   can be expected to give a better-behaved Monte Carlo estimator than~\eqref{eqloc:q*}. This is indeed what we observe numerically, see Section~\ref{sec:NetIS}. This is also related to the regularized importance scheme introduced in~\cite{korba2022adaptive}, where the importance sampling weights are set to some  powers $\alpha\in(0,1]$, for variance reduction. Indeed, the logarithm can be seen as the limit case as $\alpha\rightarrow 0$, since $(x^\alpha-1)/\alpha \to \ln x$ as $\alpha\to 0$. }

\begin{remark}[Dimension reduction bis]\label{rem:dim_redu_2} If the approximation in Remark~\ref{rem:reduction} is not correct but we have non-overlaping cluster and a decomposition $x=(y,z)\in\R^{n+m}$ between low-dimensional important variables $y$ and less informative high dimensional variables, assuming that the conditional density of $z$ given $y$ in cluster $k$ is Gaussian with law $\nu_k(\cdot|y)=\mathcal N(m_k(y),\Sigma_k(y))$, we can decompose
\begin{align*}
\partial_{p_k} J(p) &= 1+ \int_{\R^d} U \nu_k + \ln p_k + \int_{\R^d} \ln(\nu_k)\nu_k\\
& = 1+ \int_{\R^d} U \nu_k + \ln p_k + \int_{\R^n}  \co \ln(\nu_k^y) + \int_{\R^m} \nu_k(z|y)\ln \nu_k(z|y)\dd z  \cf \nu_k^y(y)\dd y\,.
\end{align*}
Instead of estimating the conditional Gaussian density of $z$ and then use~\eqref{locWk}, it is then possible to use a direct estimation of   the entropy of a Gaussian distribution, see  \cite{CAI2015161} and references therein. This only leaves a density estimation necessary for the marginal $\nu^y$ in dimension $m$.

When we use this dimension reduction, we do not need to reconstitute the variables $z$ from the variables $y$, contrary to methods which use the density estimation to design adaptive importance sampling or Metropolis-Hastings proposals \cite{schonle2025sampling,schonle2025efficient}. More generally, in our method in the non-overlapping case, the density estimation is an intermediate tool which is only used to estimate the entropy in each mode.
\end{remark}

\section{Numerical experiments}\label{sec:numerique}

\rev{
\subsection{Density estimation}\label{subsec:exdensityestim}
Before going into the description of specific experiments, let us describe how the density estimation is performed in all of them. When the dimension $d$ is small ($d=2$ in Sections~\ref{sec:exLangevin} and~\ref{sec:multiwellLangevin}), we simply use in each cluster a Gaussian kernel density estimator (specifically we used the \emph{GaussianKDEs.jl} Julia package). All experiments in higher dimension are concerned with mixtures of Gaussian distributions, with  each cluster constituted of i.i.d. samples generated according to the corresponding component of the mixture. To estimate the densities in each cluster, as discussed in Section~\ref{sec:prerequis}, we identify  a set of important variables $\xi(x)$. Here we simply take the $\ell=\min(d,10)$ coordinates with maximal variance. To estimate the marginal density of $\xi(x)$, we use in each cluster the previous Gaussian kernel density estimator. Then, when $d>\ell$, separately in each cluster, we estimate the conditional density of the remaining variables $\xi^{\perp}(x)$ given $\xi(x)$ as a Gaussian density $\mathcal N(\tilde m + \tilde \alpha \xi(x),\tilde \Sigma)$, with $\tilde \Sigma$ the empirical covariance matrix of the samples, $\tilde m$ their empirical mean and $\tilde \alpha = \tilde \Sigma_{\xi}^{-1} \tilde C$ with $\tilde \Sigma_\xi$ (resp. $\tilde C$) the empirical covariance matrix of $\xi(x)$ (resp. of $\xi(x)$ and $\xi^{\perp}(x)$), these choices being optimal in the least-square sense among densities of the form $\mathcal N(m'+\alpha' \xi(x),\Sigma')$ for $m'\in\R^{d-\ell}$, $\alpha'\in\R^{\ell\times(d-\ell)}$ and $\Sigma' \in \R^{(d-\ell)\times(d-\ell)}$.
}
\subsection{Benchmark} 

We follow the settings recently advocated by \cite{Gabrie} to evaluate the efficiency of samplers for multi-modal target distributions. The target is a Gaussian mixture
\[\mu = p_{\mathrm{true}}^* \mathcal N\po a \mathbf{1}_d, \Sigma_1\pf + \po 1 - p_{\mathrm{true}}^*\pf \mathcal N\po -a \mathbf{1}_d, \Sigma_2\pf \]
for some $a>0$, where $\mathbf{1}_d=(1,\dots,1)\in\R^d$, and the covariance matrices $\Sigma_1$ and $\Sigma_2$ are diagonal with 
\begin{equation}
\label{eq:Sigma1-2}
(\Sigma_1)_{i,i} = (\Sigma_2)_{d-i,d-i} =  \frac{d-i}{d-1} \sigma_{\mathrm{min}}^2 + \frac{i-1}{d-1} \sigma_{\mathrm{max}}^2 \,,\qquad i\in\cco 1,d\ccf,
\end{equation}
with $\sigma_{\mathrm{max}}^2= 0.2$ and $\sigma_{\mathrm{min}}^2 =0.01$.

We consider two clusters of samples, constituted by $N=1000$  i.i.d. realizations of, respectively, $\mathcal N\po a \mathbf{1}_d, \Sigma_1\pf $ and $\mathcal N\po -a \mathbf{1}_d, \Sigma_2\pf $. \rev{The density estimation follows the procedure described in Section~\ref{subsec:exdensityestim}: the variables of highest variance are  $\xi(x)=(x_{d-\ell+1},\dots,x_d)$ in the first cluster and $\xi(x) = (x_1,\dots,x_\ell)$ in the second one. }

The gradient descent~\eqref{eq1}  is run for $N_{\mathrm{iter}}=1000$ iterations, with a step-size $\delta=0.05$. It is initialized with~\eqref{eq:p*no-overlap}.

We take $p^*_{\mathrm{true}}=0.7$ and,  following~\cite{Gabrie}, we test the values $a= \{ 0.5, 2.875, 5.25, 7.625, 10\}$ and $d=\{4, 8, 16, 32, 64, 128, 256\}$, running the algorithm independently $M=48$ times for each value $(a,d)$, in order to estimate the bias and variance of its output \rev{(i.e. of $p^{N_{\mathrm{iter}}}$ as an estimator of $p_{true}^*$)}. The results are respectively provided in \rev{Tables}~\ref{fig-biais} and \ref{fig-variance}. We see that the algorithm performs very well with all these parameters. In particular, the efficiency is hardly affected by $a$ (for large $a$ in fact, the problem is even simpler since the initial guess~\eqref{eq:p*no-overlap} is already very good). Notice that comparing our results with those of the methods reviewed in \cite{Gabrie} would be unfair since in our case we consider that the modes (i.e. the position of the Gaussian distributions in the mixture) are known. \rev{For the same reason,} the model of Gaussian mixture with same weights considered in \cite{albergo2024nets,midgley2023flow} is not a relevant test model in our situation.

\begin{table}
\centering
\begin{tabular}{|c|c|c|c|c|c|c|c|c|c|c|}
\hline
\diagbox{a}{d}
 & 
 4 & 8 & 16 & 32 & 64 & 128 & 256 \\ 
\hline
$0.5   $& $5.10^{-4} 	$& $1.10^{-4} $& $	2.10^{-4}$	& $3.10^{-5} $&$ 	2.10^{-3}$ &$ 	6.10^{-3} $&$ 	2.10^{-3}$ \\ \hline
$2.875 $&  $4.10^{-4}	$& $1.10^{-3}$ 	&  $1.10^{-4}$	& $2.10^{-4} $	& $1.10^{-3}$ 	& $1.10^{-3}$ & $3.10^{-3}$\\ \hline
$5.25  $  & $3.10^{-4}$	& $2.10^{-4}	$& $2.10^{-4}	 $& $9.10^{-4}	$& $2.10^{-4}$	& $8.10^{-4}$	& $9.10^{-3}$\\ \hline
$7.625$   & $3.10^{-4}$	& $2.10^{-4}	$& $2.10^{-4}	 $& $ 7.10^{-4}$	& $6.10^{-5}	$& $1.10^{-3}$	& $1.10^{-2}$\\ \hline
$10$   &  $3.10^{-4}$	& $3.10^{-4}$ 	& $5.10^{-4}$ & $ 9.10^{-4}$	& $3.10^{-4}$ & $3.10^{-3}$    & $9.10^{-3}$\\ \hline
\end{tabular}
\caption{Bias  estimated on $M=48$ runs, for various values of $(a,d)$.}\label{fig-biais}
\end{table}

\begin{table}
\centering
\begin{tabular}{|c|c|c|c|c|c|c|c|c|c|c|}
\hline
\diagbox{a}{d}
 & 
 4 & 8 & 16 & 32 & 64 & 128 & 256 \\ 
\hline
$0.5   $&$ 1.10^{-5} 	$&$ 9.10^{-6} $&$ 	7.10^{-6}$	& $2.10^{-5} $&$ 	9.10^{-5}$ &$ 	3.10^{-4} $&$ 	1.10^{-3}$ \\ \hline
$2.875 $& $ 1.10^{-5}	$& $6.10^{-6}$ 	& $ 6.10^{-6}$	& $3.10^{-5} $	& $1.10^{-4}$ 	& $4.10^{-4}$ & $1.10^{-3}$\\ \hline
$5.25  $  & $2.10^{-5}$	& $9.10^{-6}	$& $8.10^{-6}	 $& $2.10^{-5}	$& $1.10^{-4}$	& $3.10^{-4}$	& $1.10^{-3}$\\ \hline
$7.625$   & $1.10^{-5}$	& $1.10^{-5}	$& $8.10^{-6}	 $&  $2.10^{-5}$	& $7.10^{-5}	$& $4.10^{-4}$	& $1.10^{-3}$\\ \hline
$10$   &  $1.10^{-5}$	& $7.10^{-6}$ 	&$ 1.10^{-5}$ &  $3.10^{-5}$	& $1.10^{-4}$ &$ 3.10^{-4}$    & $1.10^{-3}$\\ \hline
\end{tabular}
\caption{Variance estimated on $M=48$ runs, for various values of $(a,d)$.}\label{fig-variance}
\end{table}

\subsection{Evolution along the descent iterations}  \label{subsec:evolutiongradient}
In order to check whether a clear separation of the modes is necessary for the algorithm to work, and the respective impact of the initialization~\eqref{eq:p*no-overlap} and of the subsequent gradient descent, we perform further experiments at the fixed dimension $d=100$. All the other parameters are the same as before, except for $a,\sigma_{\mathrm{min}}^2$ and $\sigma_{\mathrm{max}}^2$ which vary. For each choice of the latter, we plot in Figure~\ref{fig-Niter} the gradient descent trajectory $p_1^m$ as a function of $m\in\cco 0,N_{\mathrm{iter}}\ccf$, initialized either at~\eqref{eq:p*no-overlap} or at $0.5$ (respectively in blue and orange). Three cases are first considered: 1) when $a=0$ and $\sigma_{\mathrm{min}}^2=\sigma_{\mathrm{min}}^2=0.2$ (up left), the samples are indistinguishable, the model is not identifiable and, as expected, the weights are not recovered; 2) when $a=0$,  $\sigma_{\mathrm{min}}^2=0.01$ and $\sigma_{\mathrm{min}}^2=0.2$ (up right), although the centers of the clusters are the same, the high-dimensional densities are very different and thus the weights are easily recovered ; 3) the choice $(a,\sigma_{\mathrm{min}}^2,\sigma_{\mathrm{max}}^2) = (0.05,0.15,0.2)$ (bottom left) shows an intermediate situation where the weights are not perfectly recovered. In these three simulations, we see that the gradient descent initialized at~\eqref{eq:p*no-overlap} hardly move along the iterations. This is because, in dimension $d=100$, with the present model, the relative entropy of one mode with respect to the other is very high as soon as $a\neq 0$ (which modifies all the coordinates) or $\sigma_{\mathrm{min}}^2\neq \sigma_{\mathrm{max}}^2$, so that the approximation discussed in Remark~\ref{rem:approximation} is very good. To illustrate a different situation, we consider a fourth case (bottom right in Figure~\ref{fig-Niter}) in dimension $d=1$ with $\sigma_{\mathrm{min}}^2= \sigma_{\mathrm{max}}^2=\sigma^2 =0.2$ and $a=0.1$. We see that in this case, the first guess~\eqref{eq:p*no-overlap} is incorrect, but the gradient descent converges to the correct value (in this simulation, convergence is not completely achieved after $1000$ iterations, so we take $N_{\mathrm{iter}}=2000$).

\begin{figure}
\begin{center}
\begin{minipage}{0.45\linewidth}
\includegraphics[width=\textwidth]{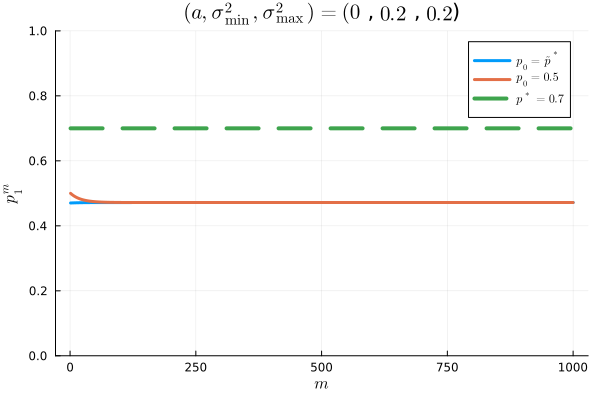}
\end{minipage}
\begin{minipage}{0.45\linewidth}
\includegraphics[width=\textwidth]{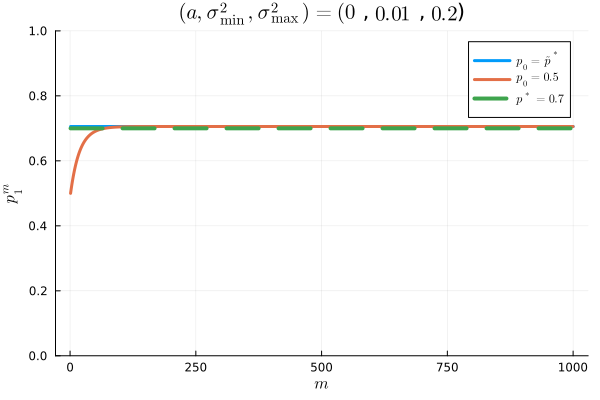}
\end{minipage}
\begin{minipage}{0.45\linewidth}
\includegraphics[width=\textwidth]{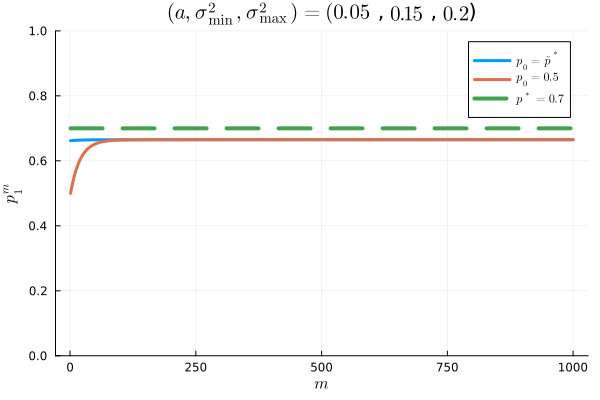}
\end{minipage}
\begin{minipage}{0.45\linewidth}
\includegraphics[width=\textwidth]{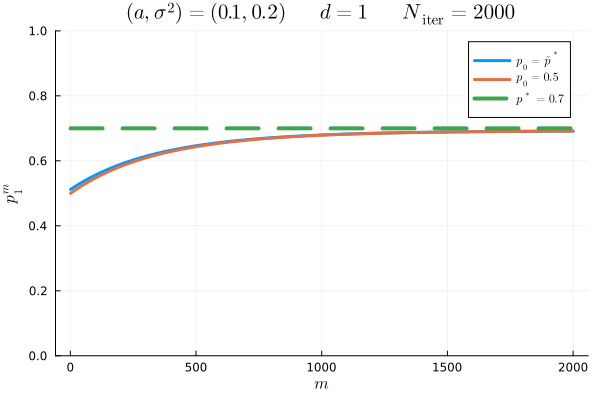}
\end{minipage}
\end{center}
\caption{Evolution of the estimator along the gradient descent. The blue line (resp. orange) is initialized at~\eqref{eq:p*no-overlap}  (resp. 0.5), and the dashed green line is the truth $p^*_{\mathrm{true}}=0.7$.  }\label{fig-Niter}
\end{figure}

\subsection{Influence of the number of samples \rev{and comparison with importance sampling}}\label{sec:NetIS}
 The third  experiment aims at assessing the effect of the approximation~\eqref{eq:estim-densite} by varying the number $N$ of samples in each mode. We take the same parameters as in Figures~\ref{fig-biais} and~\ref{fig-variance}, except that we fix $d=50$, $a=1$, and test $9$ values of $N$, spaced from $100$ to $4000$. We estimate the bias and variance of the estimator on $M=100$ independent runs. \rev{Moreover, we estimate similarly the bias and variance of the importance sampling estimation~\eqref{eqloc:q*}.} The result is given in Figure~\ref{fig:N}. We see that the error indeed decreases with $N$, and is already small when $N=500$. Notice that the model considered here is easy in terms of density estimation because each cluster simply follows a Gaussian distribution with independent coordinates (although we do not specifically use this information in the algorithm, it is such that our ``approximation" of the conditional density of the $d-\ell$ low-variance variables given the $\ell$ high-variance ones by a Gaussian distribution is in fact exact here).
 
 \rev{Furthermore, we observe that the error of our algorithm is smaller than the importance sampling algorithm (in terms of bias and also when taking into account the variance). }

\begin{figure}
\begin{center}
\includegraphics[scale=0.4]{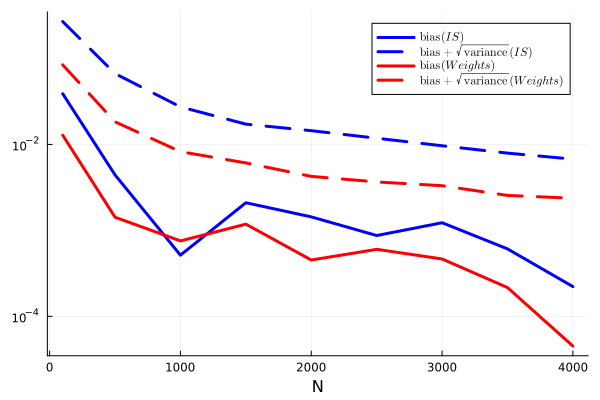}
\end{center}
\caption{Bias (plus standard deviation in dash) as a function of the number $N$ of samples per mode (in log-scale), estimated on $M=100$ runs. Red lines are for the cluster reweighting algorithm, blue lines are for the importance sampling estimation~\eqref{eqloc:q*}.}\label{fig:N}
\end{figure}

\subsection{More modes}
Fourth, to illustrate the capacity of the scheme to deal with more modes, we consider a mixture of $K=10$ Gaussian distributions, in dimension $d=100$. For the weights $p^*_{\mathrm{true}}$, we take $0.4$, $0.3$ and $0.1$ for the three first coordinates, and for the rest we take $K-3$ i.i.d. uniform random variables rescaled to sum up to $0.2$. For the parameters of each Gaussian distribution, half of them (clusters $k$ with odd $k$) have variance $\Sigma_1$ and the other half has variance $\Sigma_2$, as given by~\eqref{eq:Sigma1-2}. The means are generated as $K$ independent variables with law $\mathcal N(0,I_d)$. 

We then follow the same set-up as in the previous experiments: we sample $N=5000$ i.i.d. samples for each of the $K$ Gaussian distributions, we estimate the density with a Gaussian KDE for the $\ell=10$ variables with maximal variance (i.e. $\cco 91,100\ccf$ and $\cco 1,10\ccf$ when the variance is respectively $\Sigma_1$ and $\Sigma_2$) and the conditional density of the remaining variables is estimated with a Gaussian distribution with, \rev{as described in Section~\ref{subsec:exdensityestim}}, fixed covariance matrix and expectation given as an affine function of the $\ell$ other variables. The gradient descent~\eqref{eq1} is performed for $N_{\mathrm{iter}}=1000$ iterations with a stepsize $\delta=0.05$.

The results are provided in Figure~\ref{fig-K10}.  The first picture (left) \rev{superimposes the histograms} of the first variable in each cluster, weighted by $p^*_{\mathrm{true}}$\rev{, in the sense that the area of the histogram of $X^1$ in the cluster number $k$ is proportional to $p_k$. The weight is indicated in the label (for instance, the blue histogram has weight $p_1 = 0.4$).  That way, summing these histograms gives a good representation of the true target distribution}. The second picture of Figure~\ref{fig-K10} (right), similar to those of Figure~\ref{fig-Niter}, shows the evolution of the four first coeficients of $p^m$ along the iterations $m$ of the gradient descent, either initialized with~\eqref{eq:p*no-overlap} (solid line) or with the uniform distribution \rev{$p_k = 1/K$} (dots). The true values are given in dashed lines. We use one color per coordinate (blue for the first, red for the second, then green and yellow). We see that the order of the weights is correctly recovered (i.e. $p_1^m>p_2^m>p_3^m>p_4^m$), with moreover a good approximation for the first and fourth weights, and a relatively larger error on the second and third one. When the initialization is~\eqref{eq:p*no-overlap}, the subsequent gradient steps appear redundant.

Then, we repeat this experiment $M=48$ independent runs (as in Figure~\ref{fig-biais} and \ref{fig-variance}) to estimate the bias and variance of the estimation. The weights $p_{\mathrm{true}}^*$ and the means of the Gaussian distribution in each mode are the same in all the experiments, and the same as in Figure~\ref{fig-K10}; only the $K\times N$ samples are refreshed at each run. Denoting by $p^{(i)} \in \Delta_K$ the output of the $i^{th}$ run and $\bar p^{(M)} = \frac1M\sum_{i=1}^M p^{(i)} $, we observe
\begin{equation}
\label{eq:biaisvariance}
\left|\bar p^{(M)} - p_{\mathrm{true}}^*\right| \simeq 0.093\,,\qquad \frac{1}{M-1} \sum_{i=1}^M |p^{(i)} - \bar p^{(M)}|^2 \simeq  3. 10^{-5}\,\end{equation}
In this experiment, the bias is mostly due to the incorrect estimation of the second and third weights as in Figure~\ref{fig-K10} (right), as we observe $(\bar p^{(M)}_2,\bar p^{(M)}_3) \simeq (0.23,
0.16)$ and $|(0.23,0.16)-(0.3,0.1)| \simeq 0.092$. This means that the specific means that have been generated for these two modes have made them difficult to distinguish.

Finally, we repeat the same experiment in a more systematic way, varying $K$ and $d$, and estimating the bias and variance on $M=48$ runs as in~\eqref{eq:biaisvariance} (for each value of $(K,d)$, the weights $p_{\mathrm{true}}^*$ and the means of the Gaussian distributions are generated as above, and then are kept fixed for the $M$ runs). We test $d\in\{2,10,20,50\}$ and $K\in\{4,8,12\}$. The other parameters are as in Figure~\ref{fig-K10}, except that we only sample $N=1000$ points per Gaussian mode and $\ell = \min(d,10)$. The results are provided in Figure~\ref{fig-Kd}. It is not completely straightforward to compare the different values because for each of them the means of the target Gaussian modes are different, so that some parameters may lead to easier problems than others (if the modes are very separate for instance). However, we can see that the bias in dimension 2 is larger than in higher dimensions, and even more when $K$ is large, which is because the means of the Gaussian modes are i.i.d. $\mathcal N(0,I_d)$, so that in low dimension there is a lot of overlap and the different parts of the mixture are hard to distinguish.  Among the tested dimensions, the bias is lowest at dimension $10$, where modes are well separated and the density estimation is still easy. Then,  in higher dimensions, with only $N=1000$ sample per mode, the density estimation is less accurate. That being said, in all these settings, the results are satisfying.

\begin{figure}
\begin{center}
\begin{minipage}{0.48\linewidth}
\includegraphics[width=\textwidth]{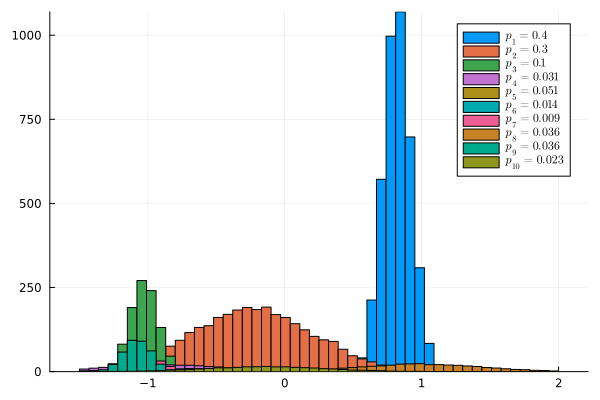}
\end{minipage}
\begin{minipage}{0.48\linewidth}
\includegraphics[width=\textwidth]{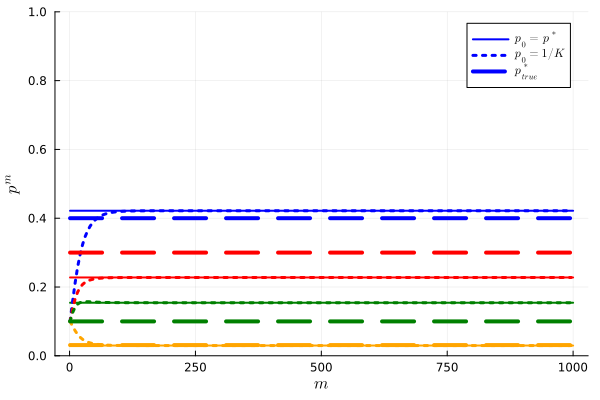}
\end{minipage}
\end{center}
\caption{Mixture of $K=10$ Gaussian distributions. Left: histograms  of the first coordinate $X_1$ in each cluster, weighted by $p^*_{true}$ (given in the label). Right: evolution of $(p_1^m,\dots,p_4^m)$ along the iterations $m$ of the gradient descent, either initialized at~\eqref{eq:p*no-overlap} (solid line) or $1/K$ (dots), with the true value $p^*_{true}$ (dashed line) for reference. The same color is used for the same coordinates (i.e. blue for the first coordinate, etc.). 
}\label{fig-K10}
\end{figure}

\begin{figure}
\centering
\[\renewcommand{\arraystretch}{1.4}
\begin{NiceArray}{|c|c|c|c|c|c|c|c|c|c|c|}
\hline
\diagbox{K}{d}
 &   2 & 10 & 20 & 50  \\ 
\hline
4   & 0.02 & 4.10^{-4}  & 0.09 & 0.2  \\ \hline
8   &  0.2 & 3.10^{-4} & 0.1 & 0.07 \\ \hline
\left.\hspace{0.1cm}\right. 12\left.\hspace{0.1cm}\right.   & 0.1  & 4.10^{-4} & 0.06 & 0.1 \\ \hline
\end{NiceArray}
\qquad
\begin{NiceArray}{|c|c|c|c|c|c|c|c|c|c|c|}
\hline
\diagbox{K    }{ d}
 &   2 & 10 & 20 & 50  \\ 
\hline
4   & 9.10^{-6} & 9.10^{-6} &1.10^{-4} & 2.10^{-4} \\ \hline
8   & 2.10^{-4} & 1.10^{-5}&  8.10^{-5} &  2.10^{-4} \\ \hline
\left.\hspace{0.1cm}\right. 12 \left.\hspace{0.1cm}\right.& 6.10^{-5} & 6.10^{-6} & 1.10^{-4} & 2.10^{-4} \\ \hline
\end{NiceArray}
\]

\caption{Bias (left) and variance (right) estimated on $M=48$ runs, varying  $(K,d)$.}\label{fig-Kd}
\end{figure}

\subsection{Langevin sampler with tempering}\label{sec:exLangevin}
 To provide a completely different situation than Gaussian mixtures, we consider the following settings. In dimension $d=2$, consider the target measure $\mu_{\beta}$ with density proportional to $e^{-\beta U}$ for some $\beta>0$ and
\begin{equation}
\label{eq:U}
U(x,y) = \frac{x^4}{4} - \frac{x^2}{2} + \frac{x^3}{5} + m(x) \frac{y^2}{2},\qquad m(x) = \frac1{10}+\frac{3}{1 + e^{2x}}\,. 
\end{equation}
Write $z=(x,y)$. As represented in Figure~\ref{fig-Langevin} (up right), this is a double-well potential, with a deep narrow well around $(-1.5,0)$ and shallow broad one around $(1,0)$. The samples are generated using $N=1000$ independent trajectories of an unadjusted Langevin sampler, with transition given by
\begin{equation}
\label{eq:EulerMaruyama}
Z_{k+1} = Z_k - h\na U(Z_k) + \sqrt{2h/\beta} G_k
\end{equation}
with i.i.d. $G_k \sim \mathcal N(0,I_2)$ and a stepsize $h=10^{-2}$. The initial conditions are distributed according to $\mathcal N(0,I_2)$.  With the objective to target the measure $\mu_{\beta_1}$ with $\beta_1=10$, we use a tempering strategy (representative of strategies used in practice, see e.g. \cite{chehab2024provable} and references therein):  we first perform $M=1000$ steps of~\eqref{eq:EulerMaruyama} (for each of the $N$ replicas) with $\beta = \beta_0 = 1$, i.e. at a higher temperature than the target \rev{ -- this can be seen as a way to perform the exploration step (i) described in the introduction}. The blue dots in Figure~\ref{fig-Langevin} (up right) are the positions of the $N$ processes after this step. Then we perform again $M$ transitions~\eqref{eq:EulerMaruyama} with the correct $\beta=\beta_1$ (which gives the white dots)\rev{, corresponding to the local sampling step (ii) in the introduction}. We apply our algorithm on these $M$ final points. We sort them in two clusters: cluster $1$ (resp. $2$) gathers the points for which the first coordinate $x$ is positive (resp. negative). 

As before, the densities $\nu_1$ and $\nu_2$ are estimated with Gaussian KDE. The result is displayed in Figure~\ref{fig-Langevin} (bottom left), which shows the points $(x_i,\nu_1(z_i))_{i\in I_1}$ in blue, $(x_i,\nu_2(z_i))_{i\in I_1}$ in orange, $(x_i,\nu_2(z_i))_{i\in I_2}$ in purple and $(x_i,\nu_1(z_i))_{i\in I_2}$ in green. As expected, since by design the two clusters have disjoint supports, the crossed terms (i.e. $\nu_k(z_i)$ with $z_i \notin I_k$) are essentially zero.

Our target is to recover the correct probability that $x>0$ under $\mu_{\beta_1}$. By design of the two clusters, our estimator~\eqref{eq:weightedestimator} is simply given by the estimated weight of the first cluster.  The result is given in Figure~\ref{fig-Langevin} (bottom right). The red dashed line is the reference (probability under $\mu_{\beta_1}$\rev{, approximately $0.03$}), computed with a deterministic quadrature. Similarly the green line is the probability under $\mu_{\beta_0}$, \rev{approximately $0.54$,} for comparison. We plot $p_1^m$ the estimator for the probability as a function of the iteration $m$ of the gradient descent, in two situations: with the initialization~\eqref{eq:p*no-overlap} (in which case we see that the gradient steps are useless, which is to be expected since by design the two clusters have separate supports) in solid blue line, and in the case $p^0 = (n_1/N,n_2/N)$ (with $n_i$ the number of points in cluster $i$; this would be the estimator without the reweighting procedure) in blue dots. We see that the initial guess~\eqref{eq:p*no-overlap}  is already very good, and that in the other situation, the estimator without reweighting\rev{, $p^0 \simeq 0.38$,} is much too close to the value at $\beta=\beta_0$ (which is not surprising since, when the temperature decreases, the particles just fall in the mode they were already and are unlikely to perform a transition afterwards)  but then the gradient descent converges quickly to the correct value.

\begin{figure}
\begin{center}
\begin{minipage}{0.45\linewidth}
\includegraphics[width=\textwidth]{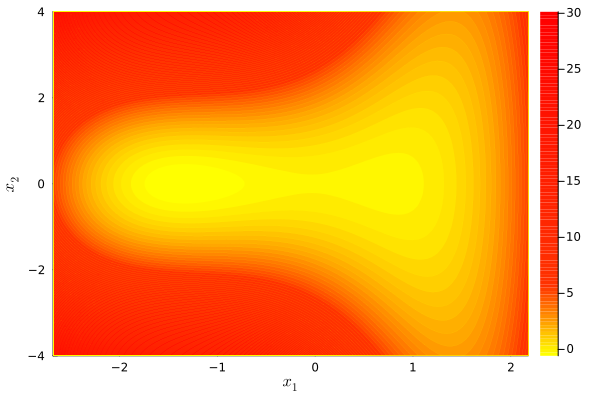}
\end{minipage}
\begin{minipage}{0.45\linewidth}
\includegraphics[width=\textwidth]{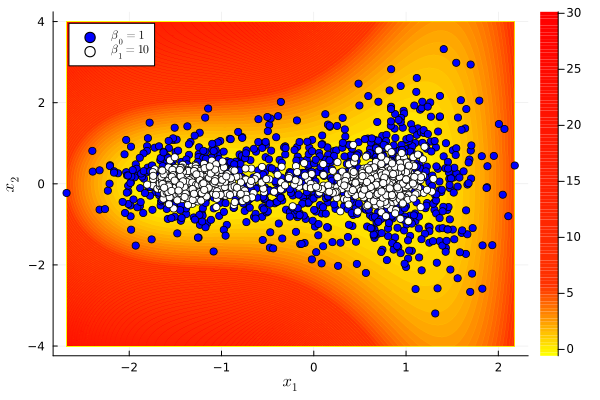}
\end{minipage}
\begin{minipage}{0.45\linewidth}
\includegraphics[width=\textwidth]{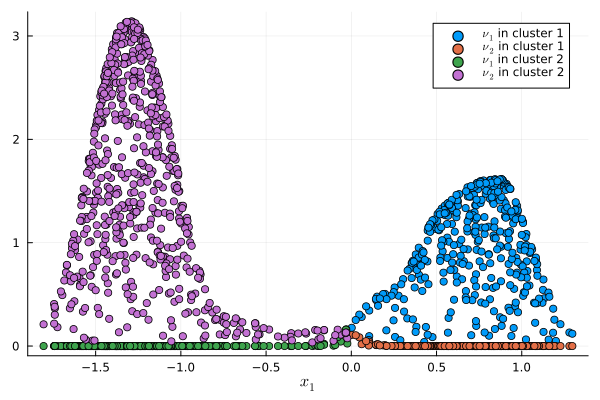}
\end{minipage}
\begin{minipage}{0.45\linewidth}
\includegraphics[width=\textwidth]{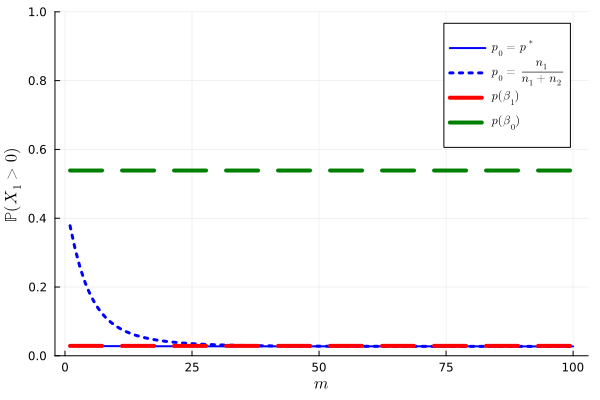}
\end{minipage}
\end{center}
\caption{Double-well tempered Langevin. Up left: potential $U$ from~\eqref{eq:U}; Up right: $N=10^3$ samples after the simulation at temperature $\beta_0^{-1}$ (blue) and then $\beta_1^{-1}$ (white); Bottom left: density estimation; Bottom right: Evolution of the estimator along the gradient descent. }\label{fig-Langevin}
\end{figure}

\rev{
\subsection{Langevin sampler with a multi-well target}\label{sec:multiwellLangevin}
In the example considered in the previous section, meant to describe a typical application of the algorithm to correct tempering strategies, the target was actually easy to describe, leading to an explicit clustering adapted to the problem. Here we consider a more realistic situation with various non-Gaussian wells of unknown positions and various shapes, with the objective to discuss the effect of the clustering step on the final output of the algorithm.}

\rev{The target measure, over $\R^2$, is $\mu_\beta \propto e^{-\beta V}$ with $\beta>0$ 	and 
\begin{equation}
\label{eq:potentielV}
V(x,y) = \frac {x^2}2 + \frac{y^2}2 + 6\cos(3x)\sin(4y+1) - 8\sin(4x-1)\cos(3y) + 4\cos(x/3-2y)\,. 
\end{equation}
This potential is represented in~Figure~\ref{fig-Langevin-multi} (up left). The  $N=10^4$  samples are obtained by i.i.d. Langevin trajectories initialised with standard Gaussian variables and performing $M=10^3$ steps of~\eqref{eq:EulerMaruyama} with step-size $h=10^{-2}$.  For clustering, we use the DBSCAN method from~\cite{ester1996density} (from the \emph{Clustering.jl} Julia package). It is an approach based on a density criterion, which makes it particularly suitable in our situation. The number of clusters is not specified by the user: rather, the output depends on a radius parameter $\varepsilon$. Small values of $\varepsilon$ provide many small clusters, larger values produce less and larger clusters. We ignore clusters of less than $100$ samples. Different results of clustering in our situation are presented in Figure~\ref{fig-Langevin-multi}.}

\begin{figure}
\begin{center}
\begin{minipage}{0.45\linewidth}
\includegraphics[width=\textwidth]{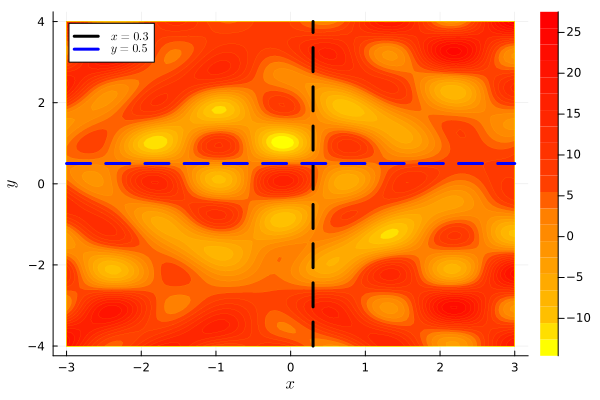}
\end{minipage}
\begin{minipage}{0.45\linewidth}
\includegraphics[width=\textwidth]{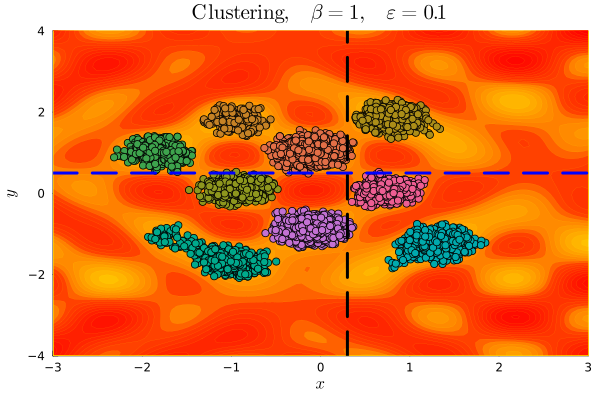}
\end{minipage}
\begin{minipage}{0.45\linewidth}
\includegraphics[width=\textwidth]{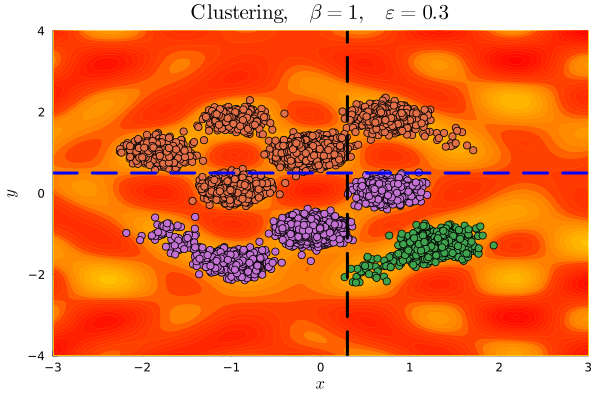}
\end{minipage}
\begin{minipage}{0.45\linewidth}
\includegraphics[width=\textwidth]{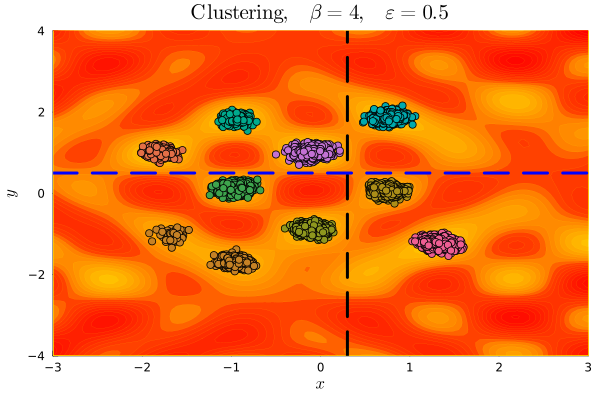}
\end{minipage}
\end{center}
\caption{Multi-well Langevin. Up left: potential $V$ from~\eqref{eq:potentielV} and the lines $\{x=0.3\}$ and $\{y=0.5\}$. Three other plots: result of the clustering step (up right, $\beta=1$ and $\varepsilon=0.1$; bottom left, $\beta=1$ and $\varepsilon=0.3$; bottom right, $\beta=4$ and $\varepsilon=0.5$).}\label{fig-Langevin-multi}
\end{figure}

\rev{Once the clusters are determined, we use a Gaussian KDE in each one to estimate the density. We initialise the gradient descent~\eqref{eq1} with~\eqref{eq:p*no-overlap} and run $m=100$ gradient steps with step-size $\delta=0.1$. To assess the efficiency of the algorithm, we consider the problem of estimating two probabilities under $\mu_\beta$, $q_x:=\mathbb P(X > 0.3 )$ and $q_y:=\mathbb P(Y > 0.5)$. These choices are made so that the lines $\{x=0.3\}$ and $\{y=0.5\}$ separates different modes, as illustrated in~Figure~\ref{fig-Langevin-multi}. Moreover, the two observables are not equally affected by an underestimation of the number of clusters. Indeed, as can be seen in the bottom left plot of Figure~\ref{fig-Langevin-multi}, in the present model, for a relatively large value of $\varepsilon$ (but not so high that there is a unique cluster), the line $\{y=0.5\}$ still roughly separates the main clusters, contrary to $\{x=0.3\}$. Since the relative weights of samples cannot be corrected within a  cluster, the reweighting procedure has then no effect on the estimation of  $q_x$  when the clusters merge modes from the different sides of the line $\{x=0.3\}$. This appears clearly in the results displayed in Figure~\ref{fig-Langevin-multi-result}, obtained as follows. For $\beta=1$, we test the values $\varepsilon = k/10$ for $k\in\cco 1,6\ccf$ in the DBSCAN algorithm. For each value, we run $M=20$ independent experiments, from which we estimate the bias and variance of the estimation of $q_x$ and $q_y$ obtained  after reweighting. The reference values of $q_x$ and $q_y$ used to estimate the bias are obtained with a deterministic quadrature. For comparison, we also provide the Monte Carlo estimation before reweighting (i.e. simply using the Langevin samples). Moreover, for each value of $\varepsilon$, we compute $N_\varepsilon$ the average number of clusters determined in the $M$ experiments, which is a more interpretable variable than $\varepsilon$.  Figure~\ref{fig-Langevin-multi-result} displays the bias (solid blue line) and the bias plus the standard devitation (dashed blue line) for the various values of $\varepsilon$, as a function of $N_\varepsilon \in \{9,8.25,7.25,6.15,4.35,1.6\}$ on the $x$-axis (here we have given the values of $N_\varepsilon$ sorted in terms of the increasing value of $\varepsilon$, namely $N_\varepsilon=9$ is obtained for $\varepsilon=1/10$), and similarly for the non-weighted estimations (solid and dashed red lines), for $q_x$ (left plot) and $q_y$ (right plot).
}

\rev{We observe that the reweighting procedure is consistently decreasing the error. When $\varepsilon$ is small enough so that all important modes are separated in different clusters (as in Figure~\ref{fig-Langevin-multi} top right), both the bias and variance are very small. This remains true if we decrease even more $\varepsilon$, up to the point where the clustering step separates all individual samples and thus there is no cluster of size larger than $100$ and no estimation is available. By contrast, for $\varepsilon \geqslant 6/10$, most of the time, there is a unique cluster, and thus the algorithm does not correct the basic MCMC estimation. For intermediate values $\varepsilon\in[3/10,6/10]$, the variance of the reweighted estimator is much larger than the variance of the basic estimation, which shows that it is not due to the local fluctuations within each mode but to the clustering step. For instance, in the estimation of $q_x$ (left plot of Figure~\ref{fig-Langevin-multi-result}) we see a drop of the variance between $\varepsilon=4/10$ and $3/10$, corresponding to $N_\varepsilon = 6.15$ and $7.25$. This is because  the probability to get a bad clustering  as in the bottom left plot of Figure~\ref{fig-Langevin-multi} (hence a bad estimation) is  negligible for $\varepsilon=3/10$ but not yet for $\varepsilon=4/10$. }

\begin{figure}
\begin{center}
\begin{minipage}{0.45\linewidth}
\includegraphics[width=\textwidth]{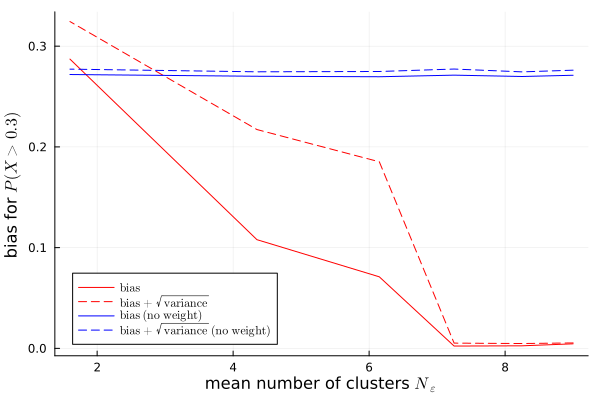}
\end{minipage}
\begin{minipage}{0.45\linewidth}
\includegraphics[width=\textwidth]{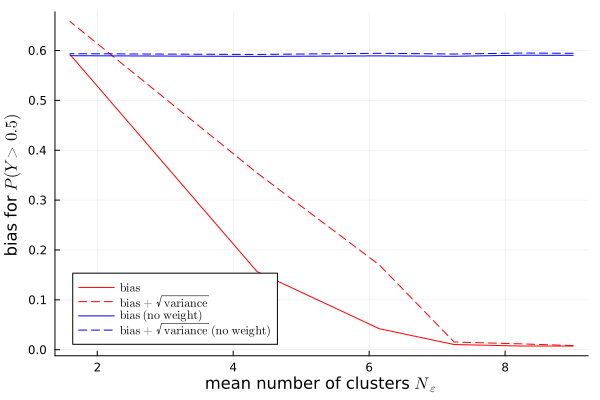}
\end{minipage}
\end{center}
\caption{Error for the estimation of $q_x$ (left) and $q_y$ (right) for different values of the DBSCAN parameter $\varepsilon$, with in the $x$-axis the average numbers of clusters. Blue lines correspond to the basis estimation without reweighting, red lines to the reweighted estimations. Solid lines are the estimated bias on $20$ runs, dashed lines are the sum of this bias and the estimated standard deviation. }\label{fig-Langevin-multi-result}
\end{figure}

\rev{For comparison, we also conduct the same experiment but for $\beta=4$. In that case, as illustrated in the bottom right plot of Figure~\ref{fig-Langevin-multi}, the distance between the modes is larger, which means that the correct number of modes is reconstituted even with larger values of $\varepsilon$ than in the case $\beta=1$. As a consequence we test the values $\varepsilon=k/10$ with $k\in\cco 3,8\ccf$. This gives averaged number of clusters $N_\varepsilon$ respectively equal to $\{9,9,8.95,8.8,5.8,2.1\}$. The results are displayed in Figure~\ref{fig-Langevin-multi-result-4}. The three values $\varepsilon\in\{0.3,0.4,0.5\}$, which all recover the correct number of clusters, are indistinguishable in the graph, with biases at most $6.10^{-4}$ (even $3.10^{-6}$ for the two smallest values of $\varepsilon$) and variance less than $7.10^{-6}$ (actually $5.10^{-13}$ for $\varepsilon\in\{0.3,0.4\}$). This highlights the fact that lowering the temperature, which would make the problem harder for Langevin sampling and other MCMC methods, actually makes the postprocessing easier as the modes get better separated, making easier their correct identification and making the optimisation problem simpler as discussed in Section~\ref{subsec:separated} and as already observed in the numerical experiments of Section~\ref{subsec:evolutiongradient}.}

\begin{figure}
\centering
\begin{center}
\begin{minipage}{0.45\linewidth}
\includegraphics[width=\textwidth]{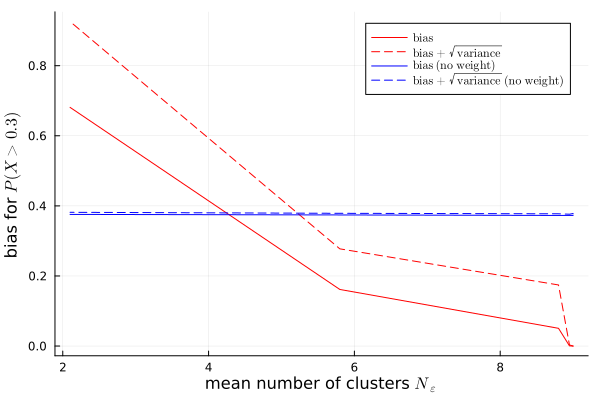}
\end{minipage}
\begin{minipage}{0.45\linewidth}
\includegraphics[width=\textwidth]{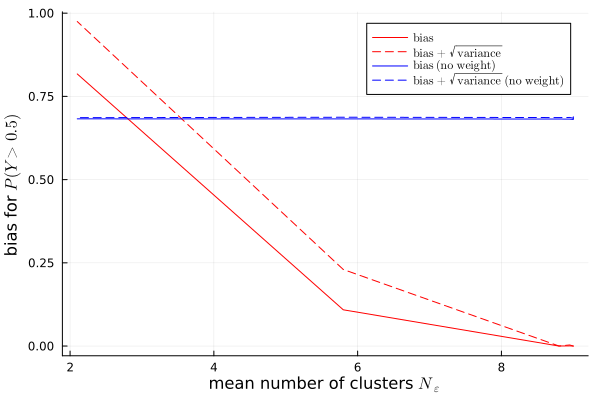}
\end{minipage}
\end{center}
\caption{Same as Figure~\ref{fig-Langevin-multi-result} but $\beta=4$. }\label{fig-Langevin-multi-result-4}
\end{figure}

\subsection{Individual weights}\label{subsec:expweight}
 To conclude, let us illustrate on a simple one-dimensional situation the case where the samples are not clustered, so that each sample $x_i$ has its own weight $p_i$ and the final estimator for the target $\mu$ is the weighted empirical distribution $ \tilde \pi(p)=\sum_{i=1}^N p_i \delta_{x_i}$. In this situation, the density estimation in each cluster makes no sense. However, we see that in our reweighting algorithm, in fact, what is necessary is to have a density estimation of $\tilde \pi(p^m)$ at each iteration $m$ of the gradient descent. Up to now we have done this by summing the densities estimated on each cluster, but in principle we can also perform a full density estimation on all the weighted samples at each iteration. Alternatively, in low dimension, a simpler approach (which we are going to follow) is to use a Gaussian kernel density with some bandwidth $\sigma$ for $\delta_{x_i}$ for each $i\in \cco 1,N\ccf$, which corresponds to the settings we have used up to now with~\eqref{eq:estim-densite} being simply $\delta_{x_i} \simeq \mathcal N(x_i,\sigma^2)$. In particular, this amounts to take $\pi(p) = \tilde \pi(p) \star \mathcal N(0,\sigma^2)$. Having in mind that what we need is the density $\pi(p)$ and not the individual densities $\mathcal N(x_i,\sigma^2)$, we scale $\sigma^2$ following the classical Sheather and Jones's method \cite{sheather1991reliable} applied to the initial empirical distribution $\tilde \pi(p^0) = \frac1N \sum_{i=1}^N \delta_{x_i}$ (with the initial condition $p^0_k = 1/N$ for all $k\in\cco 1,N\ccf$). Again, in principle we could re-compute this bandwidth for $\tilde \pi(p^m)$ at each iteration of the gradient descent but instead here we simply keep the $\sigma$ initially computed on $\tilde\pi(p^0)$).

We consider the following toy model: the samples are i.i.d. copies of $Y+Z$ where $Y \sim \mathcal N(0,1)$ and $Z\sim \mathcal U([-2,4])$, the target being $\mu = \frac12 \mathcal N(1,0.25) + \frac12 \mathcal N(-1,0.25)$. We perform $N_{\mathrm{iter}}=1000$ iterations of gradient descent with step-size $\delta=0.05$. Two experiments are performed, respectively with $N=100$ and $N=10^4$.

The results are displayed in Figure~\ref{fig-inidividuel} ($N=100$ on the left, $10^4$ on the right). As can be seen on the cumulative distribution function (bottom graphs), the method yields good results. Moreover, we see that the weights distribution (middle graphs) is smooth (close to $p_i  \propto \mu(x_i)/\nu(x_i)$ with $\nu$ the density of the data). In higher dimension, the optimization problem might be less well-conditioned and it might be necessary to add some regularization in the objective, for instance the Fisher Information of $\pi(p)$ with respect to $\mu$, or to consider a function form $p_i=\rho_{\theta}(x_i)$ for some parametrized function $\rho_{\theta}$. However, we only present this one-dimensional illustration as this variant of the algorithm is not our focus. \rev{Actually, when considering individual weights, we lose the main interest of the approach described in the rest of this work, namely that, by reweighing only well-separated modes, we get an optimisation problem which is low-dimensional, convex and well-conditioned. By contrast, the individual reweighting case is similar to the infinite dimensional VI approach for reweighting a mixture of fixed components taken in~\cite{daudel2021mixture,daudel2021infinite}, and we see in these works that the optimisation step becomes a non-trivial key difficulty.   }

\begin{figure}
\begin{center}
\begin{minipage}{0.45\linewidth}
\includegraphics[width=\textwidth]{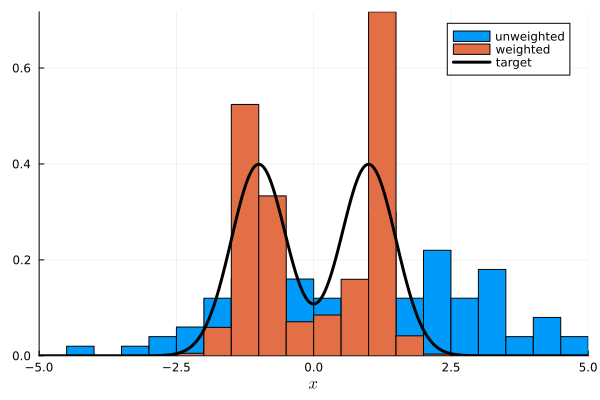}
\end{minipage}
\begin{minipage}{0.45\linewidth}
\includegraphics[width=\textwidth]{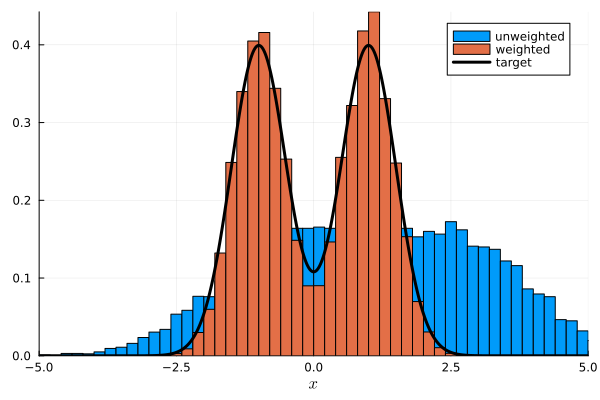}
\end{minipage}
\begin{minipage}{0.45\linewidth}
\includegraphics[width=\textwidth]{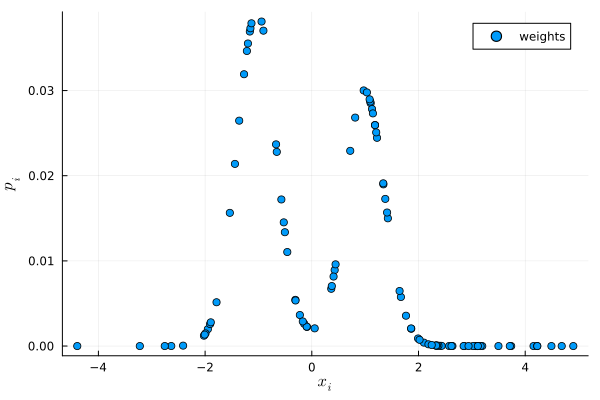}
\end{minipage}
\begin{minipage}{0.45\linewidth}
\includegraphics[width=\textwidth]{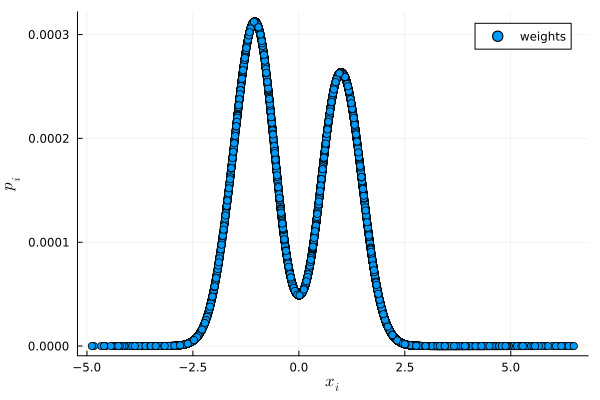}
\end{minipage}
\begin{minipage}{0.45\linewidth}
\includegraphics[width=\textwidth]{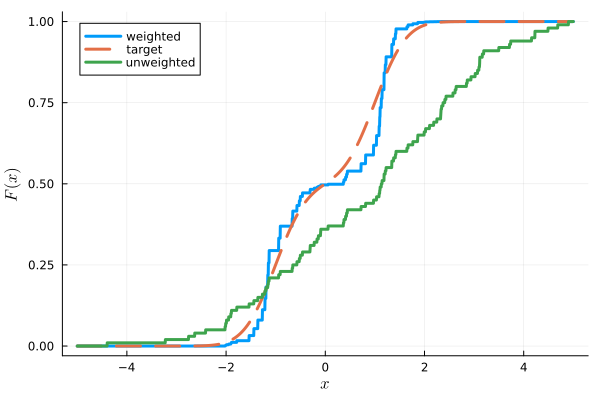}
\end{minipage}
\begin{minipage}{0.45\linewidth}
\includegraphics[width=\textwidth]{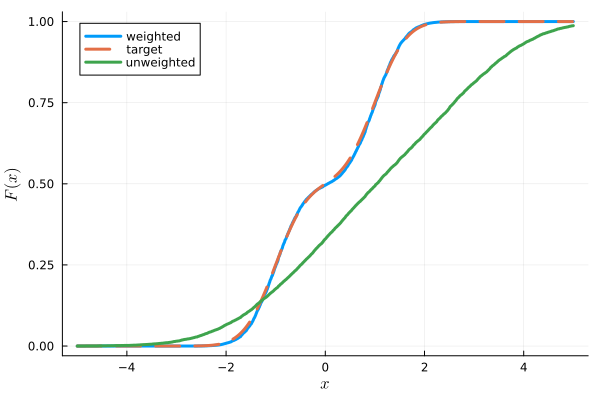}
\end{minipage}
\end{center}
\caption{Individual weights with $N$ samples (left: $N=10^2$, right: $N=10^4$). Top: initial (blue) and final reweighted (orange) data, with the target density in black. Middle: final weights $(x_i,p_i)$. Bottom: cumulative distribution function (green for the unweighted data, blue for the weithed one, and dashed orange for the target). }\label{fig-inidividuel}
\end{figure}

\section{Conclusion} 

\rev{This work presents a re-weighting procedure to correct the error in the estimation of relative weights between modes which often remains after sampling multi-modal measures. The method is similar to other existing variational inference and importance sampling approaches, the main specificities being the following: first, we consider that the samples available within each mode are good, so that we do not try to sample further nor to fit some parametric family of probability densities on them. Second, by focusing on situations of interest with a relatively small number of well-separated modes, we get to solve a well-conditioned low-dimensional convex optimisation problem, which turns out to be easy. }

All these experiments on synthetic data show that the method performs well up to dimension of the order of 100 and a dozen of modes.  Moreover, it is more efficient when the clusters are disjoint, in which case the gradient descent is not necessary and the output given by~\eqref{eq:p*no-overlap} is already a good approximation of the optimal weights. \rev{The method is thus particularly well suited for the situations for which it has initially been designed, namely when a few dozen modes are well identified and neatly separated, as it is the case in many molecular simulations (e.g. for ligand biding in biochemistry) and more generally when the three-step approach to sampling discussed in the introduction is relevant. }

\rev{This work was a proof of concept, which is why it focuses on simple situations, both for theoretical guarantees and numerical experiments. Given the encouraging results, the next step is to scale up the algorithm to address challenging applications with a high-performance user-friendly software. The main issue we wish to address in the future is the possibility to tackle higher dimensional data by developing the multi-scale approach discussed in Section~\ref{sec:prerequis} and Remarks~\ref{rem:reduction} and~\ref{rem:dim_redu_2} and applied in a very simple way in the experiments with Gaussian mixtures, as  described in Section~\ref{subsec:exdensityestim}. Going beyond the present experiments requires to chose a good method to design the re-parametrisation $(\xi,\xi^\perp)$, to explore the ways to deal with the high-dimensional variables $\xi^\perp$ and to test the impact of the dimension of the main features $\xi$ in  the performance of the method.   }

%
%
%
%
%
%
%

\bmhead{Acknowledgements}

The research of PM is supported by the project CONVIVIALITY (ANR-23-CE40-0003) of the
French National Research Agency. The author would like to thank Tony Leli{\`e}vre, Anna Korba and Marylou Gabri{\'e} for fruitful
discussions.  PM acknowledges the work of Rayan Autones who performed, during a research internship, preliminary simulations with algorithms in the same spirit as the one presented here. The total amount of generative artificial intelligence tools involved in this work is exactly zero.

\section*{Declarations}

The author has no competing interests to declare that are relevant to the content of this article


\begin{thebibliography}{10}

\bibitem{albergo2024nets}
Michael~S Albergo and Eric Vanden-Eijnden.
\newblock Nets: A non-equilibrium transport sampler.
\newblock {\em Proceedings of the 42nd International Conference on Machine Learning}, 267:1026--1055, 2025.


\bibitem{beck2003mirror}
Amir Beck and Marc Teboulle.
\newblock Mirror descent and nonlinear projected subgradient methods for convex optimization 
\newblock {\em Operations Research Letters}, 31(3):167--175, 2003.



\bibitem{belkacemi2021chasing}
Zineb Belkacemi, Paraskevi Gkeka, Tony Leli{\`e}vre, and Gabriel Stoltz.
\newblock Chasing collective variables using autoencoders and biased
  trajectories.
\newblock {\em Journal of chemical theory and computation}, 18(1):59--78, 2021.

\bibitem{bennett1976efficient}
Charles~H Bennett.
\newblock Efficient estimation of free energy differences from monte carlo
  data.
\newblock {\em Journal of Computational Physics}, 22(2):245--268, 1976.

\bibitem{biroli2026kernel}
Giulio Biroli and Marc M{\'e}zard.
\newblock Kernel density estimators in large dimensions.
\newblock {\em SIAM Journal on Mathematics of Data Science}, 8(1):46--76, 2026.

\bibitem{CAI2015161}
T.~Tony Cai, Tengyuan Liang, and Harrison~H. Zhou.
\newblock Law of log determinant of sample covariance matrix and optimal
  estimation of differential entropy for high-dimensional gaussian
  distributions.
\newblock {\em Journal of Multivariate Analysis}, 137:161--172, 2015.

\bibitem{campbell2019universal}
Trevor Campbell and Xinglong Li.
\newblock Universal boosting variational inference.
\newblock {\em Advances in Neural Information Processing Systems}, 32, 2019.

\bibitem{cecchini2009calculation}
M~Cecchini, SV~Krivov, M~Spichty, and M~Karplus.
\newblock Calculation of free-energy differences by confinement simulations.
  application to peptide conformers.
\newblock {\em The Journal of Physical Chemistry B}, 113(29):9728--9740, 2009.

\bibitem{celerse2022efficient}
Fr{\'e}d{\'e}ric C{\'e}lerse, Th{\'e}o~Jaffrelot Inizan, Louis Lagard{\`e}re,
  Olivier Adjoua, Pierre Monmarch{\'e}, Yinglong Miao, Etienne Derat, and
  Jean-Philip Piquemal.
\newblock An efficient gaussian-accelerated molecular dynamics (gamd)
  multilevel enhanced sampling strategy: application to polarizable force
  fields simulations of large biological systems.
\newblock {\em Journal of Chemical Theory and Computation}, 18(2):968--977,
  2022.

\bibitem{chehab2024provable}
Omar Chehab, Anna Korba, Austin Stromme, and Adrien Vacher.
\newblock Provable convergence and limitations of geometric tempering for
  langevin dynamics.
\newblock {\em International Conference on Learning Representations}, 2025:5451--5493, 2025.

\bibitem{chen2024new}
Wenyu Chen, Rahul Mazumder and Richard J Samworth.
\newblock A new computational framework for log-concave density estimation.
\newblock {\em Mathematical Programming Computation}, 16(2)185--228, 2024.

\bibitem{Chewietal}
Sinho Chewi,  Atsushi Nitanda and Matthew S. Zhang. 
\newblock Uniform-in-Log-Sobolev Inequality for the Mean-Field Langevin Dynamics with Convex Energy.
\newblock {\em SIAM Journal on Mathematical Analysis}, 58(4): 3652-3668, 2026.


\bibitem{Comer}
Jeffrey Comer, James~C. Gumbart, J{\'e}r{\^o}me H{\'e}nin, Tony Leli{\`e}vre,
  Andrew Pohorille, and Christophe Chipot.
\newblock {The Adaptive Biasing Force Method: everything you always wanted to
  know but were afraid to ask}.
\newblock {\em {Journal of Physical Chemistry B}}, 119(3):1129--1151, January
  2015.
  
  \bibitem{daudel2021mixture}
  Kam{\'e}lia Daudel and Randal Douc.
  \newblock {Mixture weights optimisation for alpha-divergence variational inference},
  \newblock {\em Advances in Neural Information Processing Systems}, 34:4397--4408, 2021.

\bibitem{daudel2021infinite}
 Kam{\'e}lia Daul, Randal Douc, and Fran{\c{c}}ois Portier.
 \newblock {Infinite-dimensional gradient-based descent for alpha-divergence minimisation},
\newblock {\em The Annals of Statistics}, 49(4):2250--2270, 2021.




\bibitem{DeFreitas}
Nando de~Freitas, Pedro H\o{}jen-S\o{}rensen, Michael~I Jordan, and Stuart
  Russell.
\newblock Variational mcmc.
\newblock In {\em Proceedings of the Seventeenth Conference on Uncertainty in
  Artificial Intelligence}, UAI'01, page 120?127, San Francisco, CA, USA,
  2001. Morgan Kaufmann Publishers Inc.

\bibitem{del2006sequential}
Pierre Del~Moral, Arnaud Doucet, and Ajay Jasra.
\newblock Sequential monte carlo samplers.
\newblock {\em Journal of the Royal Statistical Society Series B: Statistical
  Methodology}, 68(3):411--436, 2006.

\bibitem{domke2018importance}
Justin Domke and Daniel~R Sheldon.
\newblock Importance weighting and variational inference.
\newblock {\em Advances in neural information processing systems}, 31, 2018.

\bibitem{dutordoir2018gaussian}
Vincent Dutordoir, Hugh Salimbeni, James Hensman, and Marc Deisenroth.
\newblock Gaussian process conditional density estimation.
\newblock {\em Advances in neural information processing systems}, 31, 2018.

\bibitem{ebrahimi2022symmetry}
Mina Ebrahimi and J{\'e}r{\^o}me H{\'e}nin.
\newblock Symmetry-adapted restraints for binding free energy calculations.
\newblock {\em Journal of Chemical Theory and Computation}, 18(4):2494--2502,
  2022.
  
\bibitem{ester1996density}
Martin Ester, Hans-Peter Kriegel, J{\"o}rg Sander and Xiaowei Xu.
\newblock A density-based algorithm for discovering clusters in large spatial databases with noise.
\newblock \em{kdd}, 96(34):226--231, 1996.



\bibitem{ezugwu2022comprehensive}
Absalom~E Ezugwu, Abiodun~M Ikotun, Olaide~O Oyelade, Laith Abualigah,
  Jeffery~O Agushaka, Christopher~I Eke, and Andronicus~A Akinyelu.
\newblock A comprehensive survey of clustering algorithms: State-of-the-art
  machine learning applications, taxonomy, challenges, and future research
  prospects.
\newblock {\em Engineering Applications of Artificial Intelligence},
  110:104743, 2022.

\bibitem{Gabrie2021a}
Marylou~Gabri{\'e}, Grant~M. Rotskoff, and Eric Vanden-Eijnden.
\newblock Adaptive monte carlo augmented with normalizing flows.
\newblock {\em Proceedings of the National Academy of Sciences}, 119(10), 2022.

\bibitem{10.1214/21-EJP606}
Friedrich G{\"o}tze, Holger Sambale and Arthur Sinulis.
\newblock Concentration inequalities for polynomials in $\alpha$-sub-exponential random variables.
\newblock {\em Electronic Journal of Probability}, 26:1--22, 2021.


\bibitem{Gabrie}
Louis {Grenioux}, Maxence {Noble}, and Marylou {Gabri{\'e}}.
\newblock {Improving the evaluation of samplers on multi-modal targets}.
\newblock {\em arXiv e-prints}, page arXiv:2504.08916, April 2025.

\bibitem{hoffman2017learning}
Matthew~D Hoffman.
\newblock Learning deep latent gaussian models with markov chain monte carlo.
\newblock In {\em International conference on machine learning}, pages
  1510--1519. PMLR, 2017.

\bibitem{invernizzi2022exploration}
Michele Invernizzi and Michele Parrinello.
\newblock Exploration vs convergence speed in adaptive-bias enhanced sampling.
\newblock {\em Journal of Chemical Theory and Computation}, 18(6):3988--3996,
  2022.

\bibitem{jarzynski2002targeted}
Christopher Jarzynski.
\newblock Targeted free energy perturbation.
\newblock {\em Physical Review E}, 65(4):046122, 2002.

\bibitem{jerfel2021variational}
Ghassen Jerfel, Serena Wang, Clara Wong-Fannjiang, Katherine~A Heller, Yian Ma,
  and Michael~I Jordan.
\newblock Variational refinement for importance sampling using the forward
  kullback-leibler divergence.
\newblock In {\em Uncertainty in Artificial Intelligence}, pages 1819--1829.
  PMLR, 2021.

\bibitem{jia2020normalizing}
He~Jia and Uros Seljak.
\newblock Normalizing constant estimation with gaussianized bridge sampling.
\newblock In {\em Symposium on Advances in Approximate Bayesian Inference},
  pages 1--14. PMLR, 2020.

\bibitem{annealing}
Lucas Journel and Pierre Monmarch{\'e}.
\newblock {Convergence of the kinetic annealing for general potentials}.
\newblock {\em Electronic Journal of Probability}, 27:1 -- 37, 2022.

\bibitem{journel2023switched}
Lucas Journel and Pierre Monmarch{\'e}.
\newblock Switched diffusion processes for non-convex optimization and saddle
  points search.
\newblock {\em Statistics and Computing}, 33(6):139, 2023.

\bibitem{kivinen1997exponentiated}
Jyrki Kivinen and Manfred~K Warmuth.
\newblock Exponentiated gradient versus gradient descent for linear predictors.
\newblock {\em Information and computation}, 132(1):1--63, 1997.

\bibitem{korba2022adaptive}
Anna Korba and Fran{\c{c}}ois Portier.
\newblock {Adaptive importance sampling meets mirror descent: a bias-variance tradeoff}.
\newblock {\em International Conference on Artificial Intelligence and Statistics}, 11503--11527, 2022.



\bibitem{lange2022interpolating}
Richard~D Lange, Ari~S Benjamin, Ralf~M Haefner, and Xaq Pitkow.
\newblock Interpolating between sampling and variational inference with
  infinite stochastic mixtures.
\newblock In {\em Uncertainty in artificial intelligence}, pages 1063--1073.
  PMLR, 2022.

\bibitem{M66}
Tony {Leli{\`e}vre}, Xuyang {Lin}, and P.ierre {Monmarch{\'e}}.
\newblock {Convergence rates for an Adaptive Biasing Potential scheme from a
  Wasserstein optimization perspective}.
\newblock {\em Nonlinearity}, 39(4)045016, 2026.

\bibitem{lelievre2024using}
Tony Leli{\`e}vre and Panos Parpas.
\newblock Using witten laplacians to locate index-1 saddle points.
\newblock {\em SIAM Journal on Scientific Computing}, 46(2):A770--A797, 2024.

\bibitem{doi:10.1021/acs.jpcb.3c07075}
Tony Leli{\`e}vre, Thomas Pigeon, Gabriel Stoltz, and Wei Zhang.
\newblock Analyzing multimodal probability measures with autoencoders.
\newblock {\em The Journal of Physical Chemistry B}, 128(11):2607--2631, 2024.
\newblock PMID: 38466759.

\bibitem{midgley2023flow}
Laurence~Illing Midgley, Vincent Stimper, Gregor N.~C. Simm, Bernhard
  Sch{\"o}lkopf, and Jos{\'e}~Miguel Hern{\'a}ndez-Lobato.
\newblock Flow annealed importance sampling bootstrap.
\newblock In {\em The Eleventh International Conference on Learning
  Representations}, 2023.

\bibitem{miller2017variational}
Andrew~C Miller, Nicholas~J Foti, and Ryan~P Adams.
\newblock Variational boosting: Iteratively refining posterior approximations.
\newblock In {\em International conference on machine learning}, pages
  2420--2429. PMLR, 2017.

\bibitem{minka2001family}
Thomas~Peter Minka.
\newblock {\em A family of algorithms for approximate Bayesian inference}.
\newblock PhD thesis, Massachusetts Institute of Technology, 2001.

\bibitem{molina2024active}
Ana Molina-Taborda, Pilar Cossio, Olga Lopez-Acevedo, and Marylou Gabri{\'e}.
\newblock Active learning of boltzmann samplers and potential energies with
  quantum mechanical accuracy.
\newblock {\em Journal of Chemical Theory and Computation}, 20(20):8833--8843,
  2024.

\bibitem{ribeiro2018reweighted}
Jo{\~a}o Marcelo~Lamim Ribeiro, Pablo Bravo, Yihang Wang, and Pratyush Tiwary.
\newblock Reweighted autoencoded variational bayes for enhanced sampling
  (rave).
\newblock {\em The Journal of chemical physics}, 149(7), 2018.

\bibitem{ruiz2019contrastive}
Francisco Ruiz and Michalis Titsias.
\newblock A contrastive divergence for combining variational inference and
  mcmc.
\newblock In {\em International conference on machine learning}, pages
  5537--5545. PMLR, 2019.
  
  

\bibitem{salimans2015markov}
Tim Salimans, Diederik Kingma, and Max Welling.
\newblock Markov chain monte carlo and variational inference: Bridging the gap.
\newblock In {\em International conference on machine learning}, pages
  1218--1226. PMLR, 2015.

\bibitem{schonle2025efficient}
Christoph Sch{\"o}nle, Davide Carbone, Marylou Gabri{\'e}, Tony Leli{\`e}vre,
  and Gabriel Stoltz.
\newblock Efficient monte-carlo sampling of metastable systems using non-local
  collective variable updates.
\newblock {\em The Journal of Chemical Physics}, 164(15),  2026.

\bibitem{schonle2025sampling}
Christoph Sch{\"o}nle, Marylou Gabri{\'e}, Tony Leli{\`e}vre, and Gabriel
  Stoltz.
\newblock Sampling metastable systems using collective variables and
  jarzynski--crooks paths.
\newblock {\em Journal of Computational Physics}, 527:113806, 2025.

\bibitem{sheather1991reliable}
Simon~J Sheather and Michael~C Jones.
\newblock A reliable data-based bandwidth selection method for kernel density
  estimation.
\newblock {\em Journal of the Royal Statistical Society: Series B
  (Methodological)}, 53(3):683--690, 1991.

\bibitem{soletskyi2025theoretical}
Roman Soletskyi, Marylou Gabri{\'e}, and Bruno Loureiro.
\newblock A theoretical perspective on mode collapse in variational inference.
\newblock {\em Machine Learning: Science and Technology}, 6(2):025056, 2025.

\bibitem{trizio2024advanced}
Enrico Trizio, Andrea Rizzi, Pablo~M Piaggi, Michele Invernizzi, and Luigi
  Bonati.
\newblock Advanced simulations with plumed: Opes and machine learning
  collective variables.
\newblock {\em arXiv preprint arXiv:2410.18019}, 2024.

\bibitem{wang2019nonparametric}
Zhipeng Wang and David~W Scott.
\newblock Nonparametric density estimation for high-dimensional
  data--algorithms and applications.
\newblock {\em Wiley Interdisciplinary Reviews: Computational Statistics},
  11(4):e1461, 2019.

\bibitem{westerlund2019inflecs}
Annie~M Westerlund and Lucie Delemotte.
\newblock Inflecs: clustering free energy landscapes with gaussian mixtures.
\newblock {\em Journal of chemical theory and computation}, 15(12):6752--6759,
  2019.

\bibitem{wirnsberger2020targeted}
Peter Wirnsberger, Andrew~J Ballard, George Papamakarios, Stuart Abercrombie,
  S{\'e}bastien Racani{\`e}re, Alexander Pritzel, Danilo Jimenez~Rezende, and
  Charles Blundell.
\newblock Targeted free energy estimation via learned mappings.
\newblock {\em The Journal of Chemical Physics}, 153(14), 2020.



\bibitem{zimmermann2021nested}
Heiko Zimmermann, Hao Wu, Babak Esmaeili, and Jan-Willem van~de Meent.
\newblock Nested variational inference.
\newblock {\em Advances in Neural Information Processing Systems},
  34:20423--20435, 2021.

\end{thebibliography}
\end{document}